\documentclass[10pt]{article}

\usepackage[T1]{fontenc}
\usepackage{lmodern}

\usepackage{amsmath, amssymb, amsthm}
\usepackage[margin=2.0cm]{geometry}
\usepackage{color, graphicx, float} 
\usepackage{subcaption}
\usepackage{changepage}
\usepackage[breakable]{tcolorbox}
\usepackage{multirow}
\usepackage{tabularray}
\usepackage{caption}
\usepackage{longtable}

\usepackage{etoolbox}
\makeatletter
\patchcmd{\@maketitle}{\LARGE \@title}{\LARGE\bfseries\@title}{}{}

\renewcommand{\@seccntformat}[1]{\csname the#1\endcsname.\quad}
\makeatother

\definecolor{darkblue}{rgb}{0,0,.5}

\usepackage{hyperref}
\hypersetup{
	colorlinks=true,		
	linkcolor=darkblue,		
	citecolor=darkblue,		
	urlcolor=darkblue 		
}

\makeatletter
\def\th@plain{%
	\thm@notefont{}
	\itshape 
}
\def\th@definition{%
	\thm@notefont{}
	\normalfont 
}

\renewenvironment{proof}[1][\proofname]{\par
	\normalfont
	\topsep0\p@\@plus3\p@ \trivlist
	\item[\hskip\labelsep\itshape
	#1\@addpunct{.}]\ignorespaces
}{%
	\qed\endtrivlist
}
\makeatother

\newtheorem{theorem}{Theorem}[section]
\newtheorem{lemma}{Lemma}[section]

\newtheorem{proposition}{Proposition}[section]
\theoremstyle{definition}

\theoremstyle{definition}
\newtheorem{example}{Example}[section]
\theoremstyle{definition}
\newtheorem{remark}{Remark}[section]
\theoremstyle{definition}
\newtheorem{algorithm}{Algorithm}[section]
\theoremstyle{definition}
\newtheorem{assumption}{Assumption}[section]

\usepackage[shortlabels]{enumitem}

\newcommand{\dom}{\ensuremath{\operatorname{dom}}}

\newcommand{\argmin}{\ensuremath{\operatorname*{argmin}}}
\newcommand{\epi}{\ensuremath{\operatorname{epi}}}
\newcommand{\dist}{\ensuremath{\operatorname{dist}}}
\newcommand{\Ima}{\ensuremath{\operatorname{Im}}}

\usepackage{threeparttable}

\allowdisplaybreaks

\newcounter{step}
\newcommand\step[1]{%
	\refstepcounter{step}	
	\vskip 0.25\baselineskip
	\ifx\hfuzz#1\hfuzz
		\item[~\(\triangleright\)~\textbf{Step~\arabic{step}.}]
	\else
		\item[~\(\triangleright\)~\textbf{Step~\arabic{step}}] (\texttt{#1})\textbf{.}%
	\fi
}

\begin{document}

\title{Bregman Proximal Linearized ADMM for Minimizing Separable Sums Coupled by a Difference of Functions}

\author{
Tan Nhat Pham\thanks{Centre for New Energy Transition Research, Federation University Australia, Ballarat, VIC 3353, Australia.
E-mail: \texttt{pntan.iac@gmail.com}.},
~
Minh N. Dao\thanks{School of Science, RMIT University, Melbourne, VIC 3000, Australia.
E-mail: \texttt{minh.dao@rmit.edu.au}.},
~
Andrew Eberhard\thanks{School of Science, RMIT University, Melbourne, VIC 3000, Australia.
E-mail: \texttt{andy.eberhard@rmit.edu.au}.},
~and
Nargiz Sultanova\thanks{Centre for Smart Analytics, Federation University Australia, Ballarat, VIC 3353, Australia.
E-mail: \texttt{n.sultanova@federation.edu.au}.}
}

\date{September 30, 2024}

\maketitle

\begin{abstract}
In this paper, we develop a splitting algorithm incorporating Bregman distances to solve a broad class of linearly constrained composite optimization problems, whose objective function is the separable sum of possibly nonconvex nonsmooth functions and a smooth function, coupled by a difference of functions. This structure encapsulates numerous significant nonconvex and nonsmooth optimization problems in the current literature including the linearly constrained difference-of-convex problems. Relying on the successive linearization and alternating direction method of multipliers (ADMM), the proposed algorithm exhibits the global subsequential convergence to a stationary point of the underlying problem. We also establish the convergence of the full sequence generated by our algorithm under the Kurdyka--{\L}ojasiewicz property and some mild assumptions. The efficiency of the proposed algorithm is tested on a robust principal component analysis problem and a nonconvex optimal power flow problem.
\end{abstract}


\smallskip
\noindent{\bf Mathematics Subject Classification (MSC 2020):}
90C26,	
49M27,	
65K05.	

\section{Introduction}

Throughout this paper, the set of nonnegative integers is denoted by $\mathbb{N}$, the set of real numbers by $\mathbb{R}$, the set of nonnegative real numbers by $\mathbb{R}_+$, and the set of the positive real numbers by $\mathbb{R}_{++}$. We use $\mathbb{R}^d$ to denote the $d$-dimensional Euclidean space with inner product $\langle \cdot, \cdot \rangle$ and Euclidean norm $\|\cdot\|$. For a given matrix $M$, $\Ima(M)$ denotes
its image.

We consider the linearly constrained composite optimization problem
\begin{align}\label{eq:prob}
\min_{x_i \in \mathbb{R}^{d_i}, y\in \mathbb{R}^q} ~F(\mathbf{x},y) :=\sum_{i=1}^m f_i(x_i)+H(y)+P(\mathbf{x})-G(\mathbf{x}) \quad\text{s.t.~~}  \sum_{i=1}^{m}A_ix_i + By=b,
\end{align}
where $\mathbf{x}=(x_1,\dots,x_m)$; for each $i\in \{1,\dots,m\}$, $f_i: \mathbb{R}^{d_i} \to (-\infty,+\infty]$ is a proper lower semicontinuous function; $H:\mathbb{R}^q \to \mathbb{R}$ is a differentiable (possibly
nonconvex) function whose gradient is Lipschitz continuous; $P:\mathbb{R}^{d_1}\times \dots \times \mathbb{R}^{d_m} \to \mathbb{R}$ is another differentiable (possibly nonconvex) function whose gradient is Lipschitz continuous; $G:\mathbb{R}^{d_1}\times \dots \times \mathbb{R}^{d_m} \to \mathbb{R}$ is a continuous (possibly nonsmooth) weakly convex function, $A_i \in \mathbb{R}^{p\times d_i}$, $B\in \mathbb{R}^{p\times q}$, and $b \in \mathbb{R}^p$. Problem \eqref{eq:prob} covers two important models in the literature. Firstly, when $G \equiv 0$, problem \eqref{eq:prob} reduces to the form
\begin{align}\label{problem:WTY}
\min_{x_i \in \mathbb{R}^{d_i}, y\in \mathbb{R}^q} ~\sum_{i=1}^m f_i(x_i)+H(y)+P(\mathbf{x}) \quad\text{s.t.~~}  \sum_{i=1}^{m}A_ix_i + By=b,
\end{align}
which is studied in \cite{Wang2018}. This class of problems, together with its special cases considered in \cite{Yang2017,Guo2016}, has a broad range of useful applications in image processing, matrix decomposition, statistical learning, and others \cite{Wang2018,ADMMsurvey}. In \cite{Wang2018}, the author studied the multi-block alternating direction method of multipliers (ADMM) in the nonconvex nonsmooth setting, and they introduced new assumptions such as restricted prox-regularity and Lipschitz subminimization path to establish the global convergence. A special case of \eqref{problem:WTY} in which $P\equiv0$ was considered in \cite{Yashtini2021}, and proximal terms and dual stepsize were incorporated into the ADMM to solve the problem. An inertial proximal splitting method was also developed in \cite{XWang2023} to revisit that problem. Other simpler special cases of \eqref{problem:WTY} can be addressed by variants of the three-block ADMM  \cite{Yang2017,Lin2017,RPCA}, or two-block ADMM \cite{Guo2016,2blockBregADMM,Li_2015}, and their convergence are well studied in the literature. Secondly, when $P \equiv 0$, $m =1$, $x_1 =x \in \mathbb{R}^d$, and $A_1 =A \in \mathbb{R}^{p\times d}$, problem \eqref{eq:prob} reduces to
\begin{align}\label{DCADMM}
\min_{x \in \mathbb{R}^d,y \in \mathbb{R}^q}~ f(x)-G(x)+H(y) \quad \text{s.t.~} Ax+By=c.
\end{align}
This is known as the linearly constrained difference-of-convex problem \cite{Sun2017,Tu2019}, in which $f$, $H$, and $G$ are all assumed to be convex. Problem \eqref{DCADMM} arises from the use of the difference-of-convex regularization terms, which has shown very promising performance \cite{Gasso2009,Yin2015}. One practical engineering problem that can be solved via problem \eqref{DCADMM} is the \emph{total variation image restoration problem} \cite{Osher2005,Sun2017,Tu2019}, which is one of the most fundamental issues in imaging science. The authors of \cite{Sun2017} developed a variant of the ADMM which incorporates the Bregman distance in updating steps for solving problem \eqref{DCADMM} in the case when $f$, $G$, and $H$ are convex functions, with $H$ being differentiable with a Lipschitz continuous gradient. A hybrid Bregman ADMM was then proposed in \cite{Tu2019} to revisit problem \eqref{DCADMM} with the same assumptions on $f$, $G$, and $H$. This version introduces an additional step to evaluate the convex conjugate of $G$, allowing the flexibility in choosing the subgradient step or proximal step to update the solution by choosing a control parameter to either activate or deactivate the additional step. By this strategy, this algorithm can include the one in \cite{Sun2017} as a special case. However, further research on more general form of problem \eqref{DCADMM} is still not available in the literature, to the best of the authors' knowledge. 

Inspired by the success of those previous studies, we aim to study a class of problems that can cover both \eqref{problem:WTY} and \eqref{DCADMM}. We develop a splitting algorithm based on ADMM to solve problem \eqref{eq:prob}, which not only includes problem \eqref{problem:WTY} as a special case, but also extends problem \eqref{DCADMM} into a multi-block, nonsmooth, and nonconvex setting. Bregman distances are also incorporated into the algorithm, making it more general and allowing the flexibility in choosing the appropriate proximal terms to solve the subproblems efficiently. The sequence generated by our algorithm is bounded and each of its cluster points is a stationary point of the problem. We also prove the convergence of the full sequence under the Kurdyka––{\L}ojasiewicz property of a suitable function. 

Compared to the two closely related works in \cite{Sun2017,Tu2019}, the convexity of $f_i$ is not required, $H$ can be nonconvex, and $G$ can also be weakly convex. In other words, the convexity requirement in our work is weaker, allowing a more general class of structured optimization problems. For the convergence of the full sequence, we do not require the restricted prox-regularity assumption on $f_i$ which was used in \cite{Wang2018}. Also, although we need the assumption that $B$ has full column rank, we do not require any assumptions on the column rank of any matrices $A_i$.

In the following examples, we present two practical problems which motivate our research.

\begin{example}[Robust principal component analysis (RPCA)\cite{RPCA}]
\label{ex:RPCA}
The RPCA problem is formulated as
\begin{align*}
\min_{L,S \in \mathbb{R}^{m\times d}} \|L\|_{*} + \tau\|S\|_1  + \frac{\gamma}{2}\|L+S-M\|_F^2, 
\end{align*}
where $M$ is a given $m\times d$ observation matrix, $\|\cdot\|_{*}$ is the nuclear norm, $\|\cdot\|_1$ is the component-wise $L_1$ norm (defined as the sum of absolute values of all entries) which controls the sparsity of the solution, and $\|\cdot\|_F$ is the standard Frobenius norm which controls the noise level, where $\tau$ and $\gamma$ are positive penalty constants. By introducing a new variable $T=L+S$, the problem is rewritten as 
\begin{align*}
\min_{L,S,T \in \mathbb{R}^{m\times d}} \|L\|_{*} + \tau\|S\|_1  + \frac{\gamma}{2}\|T-M\|_F^2 \quad \text{s.t.~} T=L+S.
\end{align*}
This problem can be addressed by using the three-block ADMM as shown in \cite{RPCA}. Recently, $L_1-L_2$ regularization has been shown to have better performance than $L_1$ alone (see \cite{Lou2017} and references therein). That motivates us to investigate the effects of $\|\cdot\|_1-\|\cdot\|$ on this problem, where $\|\cdot\|$ is the  spectral norm. By using this modified regularization term, the optimization problem fits into the structure of \eqref{eq:prob}. We will later solve this problem using our proposed algorithm in Section~\ref{sec:numerical_result}.
\end{example}

\begin{example}[Direct current optimal power flow (DC-OPF) \cite{Abraham2018}]
\label{ex:DCOPF}
Optimal power flow is an important problem in power system engineering \cite{Abdi2017}. The problem's simple formulation is given as follows. Let $x_i \in \mathbb{R}^{d_i}$ be the variable associated with the $i^{th}$ node in the network. The mathematical model is then given by
\begin{align*}
\min_{x_i \in \mathbb{R}^{d_i}}~~ \sum_{i=1}^{m} \left(\frac{1}{2}x_i^\top Q_i x_i + q_i^\top x_i\right) \quad \text{s.t.~} \sum_{i=1}^{m}A_ix_i \leq b,
\end{align*}
where $Q_i$, $A_i$, $q_i$, and $b$ are coefficient matrices and vectors. The authors of \cite{Abraham2018} have successfully used multi-block ADMM to solve this problem with promising results. As modern electrical grids require the integration of renewable energy resources such as photovoltaic (PV) systems, it is necessary to find their optimal locations in the network while satisfying the power flow. This leads to the introduction of binary variables into the OPF formulation. Binary relaxation and reformulation into difference-of-convex program have been successfully used to address this problem \cite{Tan2023}. However, the algorithm proposed in \cite{Tan2023} has to rely on solvers for the subproblems. Therefore, it is tempting to investigate whether using ADMM-based algorithms can further decompose the problem and reduce the computational complexity. In Section~\ref{sec:numerical_result} we will also give the detailed formulation to show how it fits into the structure of \eqref{eq:prob}, and address it with our algorithm.
\end{example}

The rest of the paper is organized as follows. In Section~\ref{sec:preliminaries}, the preliminary materials used in this work are presented. In Section~\ref{sec:algorithm}, we introduce our algorithm, and prove the subsequential and full sequential convergence of the proposed algorithm under suitable assumptions. Section~\ref{sec:numerical_result} presents the numerical results of the algorithm. Finally, the conclusions are given in Section~\ref{sec:conclusion}.

\section{Preliminaries}
\label{sec:preliminaries}

Let $f\colon \mathbb{R}^d\to \left(-\infty,+\infty\right]$. The \emph{domain} of $f$ is $\dom f :=\{x\in \mathbb{R}^d: f(x) <+\infty\}$ and the \emph{epigraph} of $f$ is $\epi f := \{(x,\xi)\in \mathbb{R}^d\times \mathbb{R}: f(x)\leq \xi\}$. The function $f$ is said to be \emph{proper} if $\dom f \neq \varnothing$ and it never takes the value $-\infty$, \emph{lower semicontinuous} if its epigraph is closed, and \emph{convex} if its epigraph is convex. The function $f$ is said to be $\alpha$-\emph{strongly convex} ($\alpha\in \mathbb{R}_+$) if $f-\frac{\alpha}{2}\|\cdot\|^2$ is convex, and $\beta$-\emph{weakly convex} ($\beta\in \mathbb{R}_+$) if $f+\frac{\beta}{2}\|\cdot\|^2$ is convex. We say that $f$ is \emph{coercive} if $f(x)\to +\infty$ as $\|x\|\to +\infty$.

Let $x\in \mathbb{R}^d$ with $|f(x)| <+\infty$. The \emph{regular} (or \emph{Fr\'echet}) \emph{subdifferential} of $f$ at $x$ is defined by
\begin{equation*}
\widehat{\partial} f(x) :=\left\{x^*\in \mathbb{R}^d:\; \liminf_{y\to x}\frac{f(y)-f(x)-\langle x^*,y-x\rangle}{\|y-x\|}\geq 0\right\}
\end{equation*}
and the \emph{limiting} (or \emph{Mordukhovich}) \emph{subdifferential} of $f$ at $x$ is defined by
\begin{equation*}
\partial_L f(x) :=\left\{x^*\in \mathbb{R}^d:\; \exists x_n\stackrel{f}{\to}x,\; x_n^*\to x^* \text{~~with~~} x_n^*\in \widehat{\partial} f(x_n)\right\},
\end{equation*}
where the notation $y\stackrel{f}{\to}x$ means $y\to x$ with $f(y)\to f(x)$. In the case where $\lvert f(x) \rvert =+\infty$, both regular subdifferential and limiting subdifferential of $f$ at $x$ are defined to be the empty set. The \emph{domain} of $\partial_L f$ is given by $\dom \partial_L f :=\{x\in \mathbb{R}^d: \partial_L f(x) \neq \varnothing\}$. 
We now collect some important properties of the limiting subdifferential.

\begin{lemma}[Calculus rules]
\label{l:cal}
Let $f, g\colon \mathbb{R}^d \to (-\infty, +\infty]$ be proper lower semicontinuous functions, and let $x \in \mathbb{R}^d$. 
\begin{enumerate}
\item\label{l:cal_sum} 
(Sum rule) Suppose that $f$ is finite at $x$ and $g$ is locally Lipschitz around $\overline{x}$. Then $\partial_L (f + g)(\overline{x}) \subseteq \partial_L f(\overline{x})+\partial_L g(\overline{x})$. Moreover, if $g$ is strictly differentiable at $\overline{x}$, then $\partial_L(f + g)(\overline{x}) = \partial_Lf(\overline{x}) + \nabla g(\overline{x})$.
\item\label{l:cal_sep} 
(Separable sum rule) If $f(\mathbf{x})=\sum_{i=1}^{m}f_i(x_i)$ with $\mathbf{x}=(x_1,\dots,x_m)$, then $\partial_L f(\mathbf{x}) = \partial_L f_1(x_1) \times \dots \times \partial_L f_m(x_m)$. 
\end{enumerate}
\end{lemma}
\begin{proof}
\ref{l:cal_sum} follows from \cite[Proposition~1.107(ii) and Theorem~3.36]{Mordukhovich2006}, while \ref{l:cal_sep} follows from \cite[Proposition~10.5]{Rockafellar1998}.
\end{proof}

\begin{lemma}[Upper semicontinuity of subdifferential]
\label{l:upsemicont}
Let $f\colon \mathbb{R}^d \to [-\infty, +\infty]$ be Lipschitz continuous around $x \in \mathbb{R}^d$ with $|f(x)| <+\infty$, and consider sequences $(x_n)_{n\in \mathbb{N}}$ and $(x_n^*)_{n\in \mathbb{N}}$ in $\mathbb{R}^d$ such that $x_n \to x$ and, for all ${n\in \mathbb{N}}$, $x_n^* \in \partial_L f(x_n)$. 
Then $(x_n^*)_{n\in \mathbb{N}}$ is bounded and its cluster points are contained in $\partial_L f(x)$.
\end{lemma}
\begin{proof}
This follows from \cite[Lemma 2.2]{Tan2023}.
\end{proof}

We recall the concept of the well-known Bregman distance in a real Hilbert space $\mathcal{H}$, which was introduced in \cite{Bregman}. Given a differentiable convex function $\phi: \mathcal{H} \to \mathbb{R}$, the \emph{Bregman distance} associated with $\phi$ is defined as 
\begin{align*}
\forall u, v\in \mathcal{H},\quad D_{\phi}(u,v)=\phi(u)-\phi(v)-\langle \nabla\phi (v),u-v\rangle,
\end{align*}
When $\phi(x)=\|x\|^2$, then $D_{\phi}(u,v)=\|u-v\|^2$. When $\phi(x)=x^\top M x$, where $M$ is a positive semidefinite matrix, then $D_{\phi}(u,v) = \|u-v\|^2_M :=(u-v)^\top M (u-v)$. When $\mathcal{H}=\mathbb{R}^{p\times d}$ and $\phi(X)=\|X\|_F^2$, then $D_{\phi}(U,V) =\|U-V\|_F^2$ \cite{Dhillon2008}. Some useful properties of the Bregman distance are listed in the following proposition.

\begin{proposition}
\label{p:Bregman}
Let $\phi: \mathcal{H} \to \mathbb{R}$ be a differentiable convex function and $D_{\phi}$ is the Bregman distance associated with $\phi$. Then the following hold:
\begin{enumerate}
\item\label{p:Bregman_1} 
For all $u, v \in \mathcal{H}$, $D_{\phi}(u, v)\geq 0$, and $D_{\phi}(u, v)=0$ if and only if $u = v$. 

\item\label{p:Bregman_2} 
For each $v \in \mathcal{H}$, $D_{\phi}(\cdot, v)$ is convex. 

\item\label{p:Bregman_3} 
If $\phi$ is $\alpha$-strongly convex,  then, for all $u, v \in \mathcal{H}$, $D_{\phi}(u, v) \geq \frac{\alpha}{2} \|u-v\|^2$. 

\item\label{p:Bregman_4} 
If $\nabla \phi$ is $\ell_{\phi}$-Lipschitz continuous, then, for all $u, v \in \mathcal{H}$, $D_{\phi}(u,v)\leq \frac{\ell_{\phi}}{2}\|u-v\|^2$. 
\end{enumerate}
\end{proposition}
\begin{proof}
\ref{p:Bregman_1} is given in \cite[Section 1]{Bregman}. \ref{p:Bregman_2}, \ref{p:Bregman_3}, and \ref{p:Bregman_4} can be easily verified by using the definitions. 
\end{proof}

We end this section by the following lemma which will be instrumental in proving the convergence results in the next section. 
\begin{lemma}
\label{l:diff}
Let $h\colon \mathbb{R}^d\to \mathbb{R}$ be a differentiable function with $\ell$-Lipschitz continuous gradient. Then the following hold:
\begin{enumerate}
\item\label{l:diff_descent}
For all $x, y\in \mathbb{R}^d$, $|h(y) -h(x) -\langle \nabla h(x), y -x\rangle|\leq \frac{\ell}{2}\|y -x\|^2$.
\item\label{l:diff_bounded}
If $h$ is bounded below and $\ell >0$, then
\begin{align*}
\inf_{x\in \mathbb{R}^d} \left(h(x) -\frac{1}{2\ell}\|\nabla h(x)\|^2\right) >-\infty.    
\end{align*}
\end{enumerate}  
\end{lemma}
\begin{proof}
\ref{l:diff_descent}: This follows from \cite[Lemma~1.2.3]{Nesterov2018}.

\ref{l:diff_bounded}: Let $x\in \mathbb{R}^d$. Following the argument of \cite[Equation~(1.2.19)]{Nesterov2018}, we apply \ref{l:diff_descent} with $y =x -\frac{1}{\ell}\nabla h(x)$ to obtain that
\begin{align*}
h\left(x -\frac{1}{\ell}\nabla h(x)\right) &\leq h(x) +\langle \nabla h(x), y -x\rangle +\frac{\ell}{2}\|y -x\|^2 \\
&= h(x) -\frac{1}{\ell}\|\nabla h(x)\|^2 +\frac{1}{2\ell}\|\nabla h(x)\|^2 \\
&= h(x) -\frac{1}{2\ell}\|\nabla h(x)\|^2,
\end{align*}
which together with the boundedness of $h$ implies the conclusion.
\end{proof}

\section{Proposed algorithm and convergence analysis}
\label{sec:algorithm} 

We first fix some notations which are used throughout the paper from now on. Let $\mathbf{A}=[A_1|A_2|\dots|A_m]$ and let $\mathbf{I}_{q\times q}$ be the $q\times q$ identity matrix. Then $\mathbf{Ax}=\sum_{i=1}^{m}A_ix_i$. We also denote by $\partial_L^{x}$ the limiting subdifferential with respect to the $x$-variable, and by $\nabla_i f$ the $i^{th}$ block of the gradient vector $\nabla f$. Given a matrix $M$, $\lambda_{\min}(M)$ denotes its smallest eigenvalue while $\lambda_{\max}(M)$ denotes its largest eigenvalue. The following assumptions are used in our convergence analysis
\begin{assumption}[Standing assumptions]
\label{a:standing}
~
\begin{enumerate}
\item 
$\nabla H$ is $\ell_H$-Lipschitz continuous, $\nabla P$ is $\ell_P$-Lipschitz continuous, and $G$ is $\beta$-weakly convex.
\item\label{a:standing_image} 
$\Ima(\mathbf{A}) \bigcup \{b \}\subseteq \Ima(B)$.
\item\label{a:standing_rank} 
 $\lambda :=\lambda_{\min}(B^\top B) >0$ (equivalently, $B$ has full column rank).
\end{enumerate}
\end{assumption}

We note from \cite[Section~4.1]{Wang2018} that Assumption~\ref{a:standing}\ref{a:standing_image} is crucial for the convergence of ADMM methods in nonconvex settings and generally cannot be completely removed. For this assumption to hold, a sufficient condition is that $B$ has full row rank.

Recall that the augmented Lagrangian of problem \eqref{eq:prob} is given by
\begin{align*}
\mathcal{L}_\rho(\mathbf{x},y,z)=\sum_{i=1}^m f_i(x_i)+H(y)+P(\mathbf{x})-G(\mathbf{x}) +  \langle z, \mathbf{Ax}+By-b \rangle +\frac{\rho}{2}\|\mathbf{Ax}+By-b\|^2,
\end{align*}
where $z\in \mathbb{R}^p$ and $\rho\in \mathbb{R}_{++}$. Now, we present our splitting algorithm with guaranteed global convergence to a stationary point $(\overline{\mathbf{x}},\overline{y},\overline{z})$ of $\mathcal{L}_\rho$, i.e., $0 \in \partial_L \mathcal{L}_\rho(\overline{\mathbf{x}},\overline{y},\overline{z})$, or equivalently,
\begin{equation}\label{eq:optimality_conditions}
\begin{aligned}
&0 \in \partial_L\left(\sum_{i=1}^{m}f_i -G\right)(\overline{\mathbf{x}}) +\nabla P(\overline{\mathbf{x}}) +\mathbf{A}^\top \overline{z},\\
&0 =\nabla H(\overline{y}) +B^\top \overline{z}, \\
&0 =\mathbf{A\overline{x}} +B\overline{y} -b.
\end{aligned}
\end{equation}

\begin{tcolorbox}[
	left=0pt,right=0pt,top=0pt,bottom=0pt,
	colback=blue!10!white, colframe=blue!50!white,
  	boxrule=0.2pt,
  	breakable]
\begin{algorithm}[Bregman proximal linearized ADMM (BPL-ADMM)]
\label{algo:BPLA}
\step{}\label{step:A2_1}
Set $n=0$. Choose $\mathbf{x}_0 =(x_{1,0}, \dots, x_{m,0}) \in \mathbb{R}^{d_1}\times \dots \times \mathbb{R}^{d_m}$, $y_0 \in \mathbb{R}^q$, and $z_0 \in \mathbb{R}^p$. Let $\phi_i: \mathbb{R}^{d_i} \to \mathbb{R}$, $i \in \{1,\dots,m\}$, be differentiable $\alpha$-strongly convex functions with $\alpha \in (0, +\infty)$, and let $\psi: \mathbb{R}^q \to \mathbb{R}$ be a differentiable convex function with $\ell_{\psi}$-Lipschitz continuous gradient. Let 
\begin{align*}
\mu \in \left(\frac{\ell_P+\beta}{\alpha}, +\infty\right),\;\; \nu \in \left[0, +\infty\right), \text{~~and~~} \rho \in \left(\frac{\ell_H +\sqrt{\ell_H^2 +8(\ell_H +2\nu\ell_{\psi})^2}}{2\lambda}, +\infty\right).    
\end{align*}

\step{}\label{step:A2_2}
Calculate $\nabla P(\mathbf{x}_n)=(\nabla_1 P(\mathbf{x}_n),\dots,\nabla_m P(\mathbf{x}_n))$ and $g_{n}=(g_{1,n},\dots,g_{m,n}) \in \partial_L G(\mathbf{x}_{n})$. Let $\mu_n \in [\mu, +\infty)$ and $\nu_n \in [0, \nu]$. For each $i \in \{1,\dots,m\}$, define $\mathbf{u}_{i,n+1}: \mathbb{R}^{d_i}\to \mathbb{R}^{d_1}\times \dots \times \mathbb{R}^{d_m}$ by $\mathbf{u}_{i,n+1}(x_i) =(x_{1,n+1}, \dots, x_{i-1,n+1}, x_i, x_{i+1,n}, \dots, x_{m,n})$, and find
\begin{align*}
x_{i,n+1} &\in  \argmin_{x_i \in \mathbb{R}^{d_i}} \left(f_i(x_i) +\langle \nabla_i P(\mathbf{x}_n) -g_{i,n}, x_i \rangle + \langle z_n, A_ix_i\rangle+\frac{\rho}{2}\|\mathbf{A}\mathbf{u}_{i,n+1}(x_i)+By_n-b\|^2 +\mu_n D_{\phi_i}(x_i,x_{i,n})\right),\\
y_{n+1} &\in  \argmin_{y \in \mathbb{R}^q} \left(H(y) + \langle z_n, By\rangle+\frac{\rho}{2}\|\mathbf{Ax}_{n+1} + By-b \|^2 +\nu_n D_{\psi}(y,y_n) \right),\\
z_{n+1} &=z_n+\rho(\mathbf{Ax}_{n+1}+By_{n+1}-b),
\end{align*}
where $\mathbf{x}_{n+1} =(x_{1,n+1}, \dots, x_{m,n+1})$.
\step{}
If a termination criterion is not met, set $n =n+1$ and go to Step~\ref{step:A2_2}.
\end{algorithm}
\end{tcolorbox}

\begin{remark}[Discussion of the algorithm structure] Some comments on Algorithm~\ref{algo:BPLA} are in order.
\begin{enumerate}
\item 
For each $n\in \mathbb{N}$ and $g_n\in \partial_L G(\mathbf{x}_n)$, we define the linearized augmented Lagrangian as
\begin{align*}
\mathcal{L}_{\rho,n}(\mathbf{x},y,z)=\sum_{i=1}^m f_i(x_i)+H(y)+\langle \nabla P(\mathbf{x}_n)-g_{n}, \mathbf{x} \rangle +  \langle z, \mathbf{Ax}+By-b \rangle +\frac{\rho}{2}\|\mathbf{Ax}+By-b\|^2.
\end{align*}
Then the updating scheme given in Step~\ref{step:A2_2} of Algorithm~\ref{algo:BPLA} can be written as
\begin{align*}
\begin{cases}
x_{i,n+1}\!\!\! &=\argmin\limits_{x_i \in \mathbb{R}^{d_i}}\left(\mathcal{L}_{\rho,n}(\mathbf{u}_{i,n+1}(x_i),y_n,z_n) + \mu_n D_{\phi_i}(x_i,x_{i,n})\right), \\
y_{n+1} &=\argmin\limits_{y\in \mathbb{R}^q} \left(\mathcal{L}_{\rho,n}(\mathbf{x}_{n+1},y,z_n) +\nu_n D_{\psi_i}(y,y_{n})\right),\\
z_{n+1} &=z_n+\rho(\mathbf{Ax}_{n+1}+By_{n+1}-b).
\end{cases}
\end{align*}

\item 
In the case that $P \equiv 0$, $m =1$, $x_1 =x \in \mathbb{R}^d$, $A_1 =A \in \mathbb{R}^{p\times d}$, and $\mu_n =\nu_n =1$ for all $n\in \mathbb{N}$, the updates in Step~\ref{step:A2_2} of Algorithm~\ref{algo:BPLA} becomes
    \begin{align*}
    \begin{cases}
        &x_{n+1} \in  \argmin_{x \in \mathbb{R}^{d}} \left(f(x) -\langle g_n, x \rangle + \langle z_n, Ax\rangle+\frac{\rho}{2}\|Ax+By_n-b\|^2 + D_{\phi}(x,x_n)\right),\\
        &y_{n+1} \in  \argmin_{y \in \mathbb{R}^q} \left(H(y) + \langle z_n, By\rangle+\frac{\rho}{2}\|Ax_{n+1} + By-b \|^2 + D_{\psi}(y,y_n) \right),\\
        &z_{n+1}=z_n+\rho(Ax_{n+1}+By_{n+1}-b).
    \end{cases}
    \end{align*}
    In turn, Algorithm~\ref{algo:BPLA} reduces to the one introduced in \cite{Sun2017}, in which $f$, $H$, $G$ are all convex, and both $\phi$ and $\psi$ are also assumed to be strongly convex functions and have Lipschitz continuous gradients.

\item It is also worth noting that since the problems considered in \cite{Yashtini2021,Yang2017,Lin2017,2blockBregADMM,XWang2023,Guo2016,Li_2015} are all special cases of \eqref{problem:WTY}, they can also be solved by our proposed algorithm.

\item 
The performance of the ADMM heavily relies on whether the $x_i-$updates and the $y-$update can be efficiently solved or not. The ideal situation is that those updates can be written into tractable forms, such as quadratic forms or closed-form proximal operators (see Section~\ref{sec:numerical_result} for two examples). In such cases, the ADMM can converge faster than some algorithms which need external solvers. For further discussion on the various closed-form proximal operators, we refer the readers to \cite[Remark 3.1]{Tan2023}. Incorporating Bregman distance into ADMM also has many benefits in terms of flexibility and performance, as shown in \cite{Li_2015,Sun2017,Tu2019,2blockBregADMM}, when appropriate choices of Bregman distance are used.

\end{enumerate}
\end{remark}

The following lemmas will be useful for our convergence analysis.

\begin{lemma}
\label{l:norm}
Let $L$ be a symmetric matrix in $\mathbb{R}^{d\times d}$ and let $M$ be a matrix in $\mathbb{R}^{p\times d}$. Then the following hold:
\begin{enumerate}
\item\label{l:norm_Mx}
For all $x\in \mathbb{R}^d$, $\sqrt{\lambda_{\min}(M^\top M)}\|x\| \leq \|Mx\|$.
\item\label{l:norm_M'x}
For all $x\in \mathbb{R}^d$, $\gamma(x^\top Lx) \leq \|Lx\|^2$, where $\gamma =\min\{|\lambda|: \lambda \text{~is a nonzero eigenvalue of~} L\}$. Consequently, for all $z\in \Ima(M)$, $\sqrt{\lambda_*(M^\top M)}\|z\| \leq \|M^\top z\|$, where $\lambda_*(M^\top M)$ is the smallest strictly positive eigenvalue of $M^\top M$. 
\end{enumerate}
\end{lemma}
\begin{proof}
\ref{l:norm_Mx}: This follows from the fact that $\|Mx\|^2 =x^\top (M^\top M)x$.

\ref{l:norm_M'x}: 
The first assertion follows from \cite[Lemma~4.5]{BDL23} and its proof. Now, let $z\in \Ima(M)$. Then there exists $x\in \mathbb{R}^d$ such that $z =Mx$. Since $M^\top M$ is a symmetric matrix in $\mathbb{R}^{d\times d}$ with all eigenvalues being nonnegative, we have from the first assertion that
\begin{align*}
\lambda_*(M^\top M)\|z\|^2 =\lambda_*(M^\top M)(x^\top M^\top Mx) \leq \|M^\top Mx\|^2 =\|M^\top z\|^2,    
\end{align*}
which completes the proof.
\end{proof}

We move on to prove the main results of the paper, starting with the guaranteed subsequential convergence. 

\begin{theorem}[Subsequential convergence]
\label{t:cvg}
Suppose that Assumption~\ref{a:standing} holds. Let $(\mathbf{x}_n,y_n,z_n)_{n\in \mathbb{N}}$ be the sequence generated by Algorithm~\ref{algo:BPLA} and, for each $n\in \mathbb{N}^*$, define
\begin{align*}
L_n =\mathcal{L}_\rho(\mathbf{x}_{n},y_{n},z_{n}) +c\|y_n-y_{n-1}\|^2, \text{~where~} c=\frac{(\ell_H+2\nu\ell_{\psi})\nu\ell_{\psi}}{\lambda\rho}.     
\end{align*}
Then the following hold:
\begin{enumerate}
\item\label{t:cvg_decrease}
For all $n\in \mathbb{N}^*$,
\begin{align}\label{eq:decrease}
L_{n+1} +\delta_{\mathbf{x}}\|\mathbf{x}_{n+1}-\mathbf{x}_{n}\|^2 +\delta_y\|y_{n+1}-y_n\|^2 \leq L_n,
\end{align}
where $\delta_{\mathbf{x}} =\frac{\mu\alpha -\ell_P -\beta}{2} > 0$ and $\delta_y =\frac{\lambda\rho}{2} -\frac{(\ell_H+2\nu\ell_{\psi})^2}{\lambda\rho} -\frac{\ell_H}{2} >0$.

\item\label{t:cvg_seq}
Suppose that $z_0\in \Ima(B)$, that $\sum_{i=1}^m f_i(x_i)+P(\mathbf{x})-G(\mathbf{x})$ is coercive, and that $H$ is bounded below. 
Then the sequence $(\mathbf{x}_n,{y}_n,{z}_n)_{n\in \mathbb{N}}$ is bounded and $\mathbf{x}_{n+1}-\mathbf{x}_n\to 0$, $y_{n+1}-y_n\to 0$, $z_{n+1}-z_n\to 0$, and $\mathbf{Ax}_n +By_n \to b$ as $n\to +\infty$. Furthermore, for every cluster point $(\overline{\mathbf{x}},\overline{y},\overline{z})$ of $(\mathbf{x}_n,{y}_n,{z}_n)_{n\in \mathbb{N}}$, it holds that
\begin{align*}
\mathbf{A}\overline{\mathbf{x}} +b\overline{y} =b \text{~~and~~} \lim_{n \to +\infty} F(\mathbf{x}_{n },y_{n}) =\lim_{n \to +\infty}\mathcal{L}_\rho(\mathbf{x}_{n },y_{n},z_{n}) =\mathcal{L}_\rho(\overline{\mathbf{x}},\overline{y},\overline{z}) =F(\overline{\mathbf{x}},\overline{y}),    
\end{align*}
and that if $G$ is strictly differentiable and, for each $i \in \{1,\dots,m\}$, $\nabla \phi_i$ is $\ell_{\phi}$-Lipschitz continuous, then $(\overline{\mathbf{x}},\overline{y},\overline{z})$ is a stationary point of $\mathcal{L}_\rho$.
\end{enumerate}
\end{theorem}
\begin{proof}
\ref{t:cvg_decrease}: Let $n\in \mathbb{N}$. The update of $z_{n+1}$ yields
\begin{align}\label{eq:z-update}
z_{n+1} -z_n =\rho(\mathbf{Ax}_{n+1}+By_{n+1}-b).   
\end{align}
By the optimality condition for the update of $y_{n+1}$ in Step~\ref{step:A2_2} of Algorithm~\ref{algo:BPLA}, 
\begin{align*}
0 &= \nabla H(y_{n+1}) + B^\top z_n +\rho B^\top(\mathbf{Ax}_{n+1} +By_{n+1} -b) +\nu_n(\nabla \psi(y_{n+1}) - \nabla \psi(y_n)),
\end{align*}
which combined with \eqref{eq:z-update} yields
\begin{align}\label{eq:B'z}
B^\top z_{n+1} =-\nabla H(y_{n+1}) -\nu_n(\nabla \psi(y_{n+1}) -\nabla \psi(y_n)).  
\end{align}
From the definition of the augmented Lagrangian, we have that
\begin{align}\label{eq:zz}
\mathcal{L}_\rho(\mathbf{x}_{n+1},y_{n+1},z_{n+1}) -\mathcal{L}_\rho(\mathbf{x}_{n+1},y_{n+1},z_{n}) &= \langle z_{n+1}-z_n, \mathbf{Ax}_{n+1}+By_{n+1}-b \rangle \notag\\
&= \frac{1}{\rho}\|z_{n+1}-z_n\|^2 
\end{align}
and that
\begin{align*}
\mathcal{L}_\rho(\mathbf{x}_{n+1},y_{n+1},z_n) -\mathcal{L}_\rho(\mathbf{x}_{n+1},y_n,z_n) &= H(y_{n+1}) - H(y_n) + \langle z_n, B(y_{n+1}-y_n) \rangle \\
&\qquad + \frac{\rho}{2}(\|\mathbf{Ax}_{n+1}+By_{n+1}-b\|^2 - \|\mathbf{Ax}_{n+1}+By_{n}-b\|^2).  
\end{align*}
Writing $\mathbf{Ax}_{n+1}+By_{n}-b =(\mathbf{Ax}_{n+1}+By_{n+1}-b) -(By_{n+1}- By_n)$, we obtain that
\begin{align*}
&\|\mathbf{Ax}_{n+1}+By_{n+1}-b\|^2 - \|\mathbf{Ax}_{n+1}+By_{n}-b\|^2 \\
&= -\|B(y_{n+1}-y_n)\|^2 +2\langle  \mathbf{Ax}_{n+1}+By_{n+1}-b, B(y_{n+1}-y_n)\rangle \\
&= -\|B(y_{n+1}-y_n)\|^2 +\frac{2}{\rho}\langle  z_{n+1}-z_{n}, B(y_{n+1}-y_n)\rangle,
\end{align*}
and so
\begin{align*}
\mathcal{L}_\rho(\mathbf{x}_{n+1},y_{n+1},z_n) -\mathcal{L}_\rho(\mathbf{x}_{n+1},y_n,z_n) = H(y_{n+1}) - H(y_n) +\langle z_{n+1}, B(y_{n+1}-y_n) \rangle -\frac{\rho}{2}\|B(y_{n+1}-y_n)\|^2.
\end{align*}
In view of \eqref{eq:B'z},
\begin{align*}
\langle z_{n+1}, B(y_{n+1}-y_n) \rangle &=\langle B^\top z_{n+1}, y_{n+1}-y_n \rangle \\
&=-\langle \nabla H(y_{n+1}), y_{n+1}-y_n \rangle -\nu_n\langle \nabla \psi(y_{n+1})- \nabla \psi(y_n), y_{n+1}-y_n \rangle \\
&\leq -H(y_{n+1}) +H(y_n) +\frac{\ell_H}{2}\|y_{n+1}-y_n\|^2,
\end{align*}
where the last inequality is from the $\ell_H$-Lipschitz continuity of $\nabla H$ and Lemma~\ref{l:diff}\ref{l:diff_descent}, and from the monotonicity of $\nabla \psi$. 
We therefore obtain that
\begin{align}\label{eq:yy}
\mathcal{L}_\rho(\mathbf{x}_{n+1},y_{n+1},z_n) -\mathcal{L}_\rho(\mathbf{x}_{n+1},y_n,z_n) \leq \frac{\ell_H}{2}\|y_{n+1}-y_n\|^2 -\frac{\rho}{2}\|B(y_{n+1}-y_n)\|^2.    
\end{align}
Next, by the update of $x_{i,n+1}$ in Step~\ref{step:A2_2} of Algorithm~\ref{algo:BPLA}, for all $i \in \{1,\dots,m\}$ and all $x_i \in \mathbb{R}^{d_i}$,
\begin{align}\label{eq:update_step_x}
&f_i(x_{i,n+1}) +\langle \nabla_i P(\mathbf{x}_n) -g_{i,n}, x_{i,n+1} \rangle +\langle z_n, A_{i}x_{i,n+1} \rangle +\frac{\rho}{2}\|\mathbf{A}\mathbf{u}_{i,n+1}(x_{i,n+1}) +By_n -b\|^2 +\mu_n D_{\phi_i}(x_{i,n+1},x_{i,n}) \notag \\
&\leq  f_i(x_i) +\langle \nabla_i P(\mathbf{x}_n) -g_{i,n}, x_i \rangle +\langle z_n, A_{i}x_i\rangle +\frac{\rho}{2}\|\mathbf{A}\mathbf{u}_{i,n+1}(x_i) +By_n -b\|^2 +\mu_n D_{\phi_i}(x_i,x_{i,n}).
\end{align}
Letting $x_i =x_{i,n}$, using Proposition~\ref{p:Bregman}\ref{p:Bregman_1}\&\ref{p:Bregman_3}, and summing up over $i \in \{1,\dots,m\}$, we derive that
\begin{align}\label{eq:from x-update}
&\sum_{i=1}^m f_i(x_{i,n+1}) +\langle \nabla P(\mathbf{x}_n) -g_n, \mathbf{x}_{n+1} \rangle +\langle z_n, \mathbf{Ax}_{n+1} \rangle +\frac{\rho}{2}\|\mathbf{A}\mathbf{x}_{n+1} +By_n -b\|^2 +\frac{\mu\alpha}{2}\|\mathbf{x}_{n+1}-\mathbf{x}_{n}\|^2 \notag\\
&\leq \sum_{i=1}^m f_i(x_{i,n}) +\langle \nabla P(\mathbf{x}_n) -g_n, \mathbf{x}_{n} \rangle +\langle z_n, \mathbf{Ax}_{n} \rangle +\frac{\rho}{2}\|\mathbf{A}\mathbf{x}_n +By_n -b\|^2.
\end{align}
On the other hand,
\begin{align}\label{eq:xx}
&\mathcal{L}_\rho(\mathbf{x}_{n+1},y_{n},z_n) -\mathcal{L}_\rho(\mathbf{x}_{n},y_n,z_n) \notag\\
&= \sum_{i=1}^m f_i(x_{i,n+1}) -\sum_{i=1}^m f_i(x_{i,n}) +P(\mathbf{x}_{n+1}) -P(\mathbf{x}_{n}) -G(\mathbf{x}_{n+1}) +G(\mathbf{x}_{n}) \notag\\ 
&\quad +\langle z_n, \mathbf{Ax}_{n+1}-\mathbf{Ax}_{n}\rangle +\frac{\rho}{2}\|\mathbf{Ax}_{n+1}+By_{n}-b\|^2 -\frac{\rho}{2}\|\mathbf{Ax}_{n}+By_{n}-b\|^2 \notag\\
&\leq \sum_{i=1}^m f_i(x_{i,n+1}) -\sum_{i=1}^m f_i(x_{i,n}) +\langle \nabla P(\mathbf{x}_n), \mathbf{x}_{n+1}-\mathbf{x}_n \rangle + \frac{\ell_P}{2}\|\mathbf{x}_{n+1}-\mathbf{x}_n\|^2 - \langle g_n, \mathbf{x}_{n+1} -\mathbf{x}_{n} \rangle + \frac{\beta}{2}\|\mathbf{x}_{n+1}-\mathbf{x}_{n}\|^2 \notag\\
&\quad +\langle z_n, \mathbf{Ax}_{n+1}-\mathbf{Ax}_{n}\rangle  + \frac{\rho}{2}\|\mathbf{Ax}_{n+1}+By_{n}-b\|^2 - \frac{\rho}{2}\|\mathbf{Ax}_{n}+By_{n}-b\|^2 \notag\\
&\leq -\frac{\mu\alpha -\ell_P -\beta}{2}\|\mathbf{x}_{n+1}-\mathbf{x}_{n}\|^2 =-\delta_{\mathbf{x}}\|\mathbf{x}_{n+1}-\mathbf{x}_{n}\|^2,
\end{align}
where the first inequality is obtained by using Lemma~\ref{l:diff}\ref{l:diff_descent} on the $\ell_P$-Lipschitz continuity of $\nabla P$ and \cite[Lemma~4.1]{BDL22} on the $\beta$-weak convexity of $G$, while the last inequality follows from \eqref{eq:from x-update}. Summing up three relations \eqref{eq:zz}, \eqref{eq:yy}, and \eqref{eq:xx} yields
\begin{align}\label{eq:LL}
&\mathcal{L}_\rho(\mathbf{x}_{n+1},y_{n+1},z_{n+1}) - \mathcal{L}_\rho(\mathbf{x}_{n},y_{n},z_{n}) \notag\\
&\leq -\delta_{\mathbf{x}}\|\mathbf{x}_{n+1}-\mathbf{x}_{n}\|^2 +\frac{\ell_H}{2}\|y_{n+1}-y_n\|^2 -\frac{\rho}{2}\|B(y_{n+1}-y_n)\|^2 +\frac{1}{\rho}\|z_{n+1}-z_n\|^2.
\end{align}

Now, since $\Ima(\mathbf{A}) \cup \{b\} \subseteq \Ima(B)$, it follows from \eqref{eq:z-update} that
\begin{align}\label{eq:zImB}
z_{n+1}-z_n \in \Ima(B)    
\end{align}
which, by Lemma~\ref{l:norm}\ref{l:norm_M'x}, yields 
\begin{align*}
\|z_{n+1} -z_n\| \leq \frac{1}{\sqrt{\lambda}}\|B^\top z_{n+1} -B^\top z_n\|.
\end{align*}
Let $n\in \mathbb{N}^*$. We derive from \eqref{eq:B'z} and the Lipschitz continuity of $\nabla H$ and $\nabla \psi$ that
\begin{align}\label{eq:BTZy}
\|B^\top z_{n+1} -B^\top z_n\| &\leq \|\nabla H(y_{n+1}) -\nabla H(y_n)\| +\nu_n\|\nabla \psi(y_{n+1}) -\nabla \psi(y_n)\| +\nu_{n-1}\|\nabla \psi(y_n) - \nabla \psi(y_{n-1})\| \notag \\ 
&\leq (\ell_H +\nu\ell_{\psi})\|y_{n+1} -y_n\| +\nu\ell_{\psi}\|y_n -y_{n-1}\|,
\end{align}
and thus, 
\begin{align}\label{eq:zy}
\|z_{n+1} -z_n\| \leq \frac{1}{\sqrt{\lambda}}\|B^\top z_{n+1} -B^\top z_n\| \leq \frac{1}{\sqrt{\lambda}}((\ell_H +\nu\ell_{\psi})\|y_{n+1} -y_n\| +\nu\ell_{\psi}\|y_n -y_{n-1}\|).
\end{align}
Applying Cauchy--Schwarz inequality to two vectors $u =(\sqrt{\ell_H+\nu\ell_{\psi}},\sqrt{\nu\ell_{\psi}})$ and $v =(\sqrt{\ell_H+\nu\ell_{\psi}}\|y_{n+1} -y_n\|,\sqrt{\nu\ell_{\psi}}\|y_n -y_{n-1}\|)$, we obtain that
\begin{align*}
\|z_{n+1} -z_n\|^2 \leq \frac{\ell_H +2\nu\ell_{\psi}}{\lambda}((\ell_H +\nu\ell_{\psi})\|y_{n+1} -y_n\|^2 +\nu\ell_{\psi}\|y_n -y_{n-1}\|^2).    
\end{align*}
By combining this with \eqref{eq:LL} and using Lemma~\ref{l:norm}\ref{l:norm_Mx},
\begin{align*}
&\mathcal{L}_\rho(\mathbf{x}_{n+1},y_{n+1},z_{n+1}) - \mathcal{L}_\rho(\mathbf{x}_{n},y_{n},z_{n}) \notag\\
&\leq -\delta_{\mathbf{x}}\|\mathbf{x}_{n+1}-\mathbf{x}_{n}\|^2 +\left(\frac{(\ell_H+2\nu\ell_{\psi})(\ell_H+\nu\ell_{\psi})}{\lambda\rho} +\frac{\ell_H}{2}\right)\|y_{n+1}-y_n\|^2 \\ 
&\qquad -\frac{\rho}{2}\|B(y_{n+1}-y_n)\|^2 +\frac{(\ell_H+2\nu\ell_{\psi})\nu\ell_{\psi}}{\lambda\rho}\|y_n-y_{n-1}\|^2 \\
&\leq -\delta_{\mathbf{x}}\|\mathbf{x}_{n+1}-\mathbf{x}_{n}\|^2 -\left(\frac{\lambda\rho}{2} -\frac{(\ell_H+2\nu\ell_{\psi})(\ell_H+\nu\ell_{\psi})}{\lambda\rho} -\frac{\ell_H}{2}\right)\|y_{n+1}-y_n\|^2 \\
&\qquad +\frac{(\ell_H+2\nu\ell_{\psi})\nu\ell_{\psi}}{\lambda\rho}\|y_n-y_{n-1}\|^2,
\end{align*}
or equivalently,
\begin{align*}
&\mathcal{L}_\rho(\mathbf{x}_{n+1},y_{n+1},z_{n+1}) +\frac{(\ell_H +2\nu\ell_{\psi})\nu\ell_{\psi}}{\lambda\rho}\|y_{n+1}-y_{n}\|^2 + \delta_{\mathbf{x}}\|\mathbf{x}_{n+1}-\mathbf{x}_{n}\|^2 +\delta_y\|y_{n+1}-y_n\|^2 \\
&\leq \mathcal{L}_\rho(\mathbf{x}_{n},y_{n},z_{n}) +\frac{(\ell_H +2\nu\ell_{\psi})\nu\ell_{\psi}}{\lambda\rho}\|y_n-y_{n-1}\|^2,
\end{align*}
which proves \eqref{eq:decrease}. Here, we note that $\delta_{\mathbf{x}} =\frac{\mu\alpha -\ell_P -\beta}{2} > 0$ since $\mu > \frac{\ell_P+\beta}{\alpha}$ and that $\delta_y =\frac{\lambda\rho}{2} -\frac{(\ell_H+2\nu\ell_{\psi})^2}{\lambda\rho} -\frac{\ell_H}{2} >0$ since $\rho > \frac{\ell_H +\sqrt{\ell_H^2 +8(\ell_H +2\nu\ell_{\psi})^2}}{2\lambda}$.

\ref{t:cvg_seq}: Let $n\in \mathbb{N}^*$. As $z_0 \in \Ima(B)$, we have from \eqref{eq:zImB} that $z_n \in \Ima (B)$. Using Lemma~\ref{l:norm}\ref{l:norm_M'x}, \eqref{eq:B'z}, and the Lipschitz continuity of $\psi$ yields
\begin{align*}
\sqrt{\lambda}\|z_n\|\leq \|B^\top z_n\| \leq\|\nabla H(y_n)\|+ \nu_n\|\nabla \psi(y_{n}) -\nabla \psi(y_{n-1})\| 
\leq \|\nabla H(y_n)\|+ \nu \ell_{\psi}\|y_{n}-y_{n-1}\|.
\end{align*}
For $\omega\in \mathbb{R}_{++}$ to be chosen later, invoking Cauchy--Schwarz inequality with two vectors $u =(\sqrt{\omega},\sqrt{\nu\ell_{\psi}})$ and $v =(\frac{1}{\sqrt{\omega}}\|\nabla H(y_n)\|,\sqrt{\nu\ell_{\psi}}\|y_n-y_{n-1}\|)$, we have that
\begin{align}\label{eq:uppper_zn}
\|z_n\|^2 &\leq \frac{\omega +\nu\ell_{\psi}}{\lambda} \left(\frac{1}{\omega}\|\nabla H(y_n)\|^2 +\nu\ell_{\psi}\|y_n-y_{n-1}\|^2\right) \notag \\
&= \frac{1}{\lambda}\left(1 +\frac{\nu\ell_{\psi}}{\omega}\right)\|\nabla H(y_n)\|^2 +\frac{(\omega +\nu\ell_{\psi})\nu\ell_{\psi}}{\lambda}\|y_n-y_{n-1}\|^2.
\end{align}
Combining this with the fact that
\begin{align*}
\langle z_n, \mathbf{Ax}_n+By_n-b \rangle +\frac{\rho}{2}\|\mathbf{Ax}_n+By_n-b\|^2 =\frac{\rho}{2}\left\| \frac{z_n}{\rho} +\mathbf{Ax}_n+By_n-b \right\|^2 -\frac{1}{2\rho}\|z_n\|^2,
\end{align*}
we deduce that
\begin{align*}
L_n &=\mathcal{L}_\rho(\mathbf{x}_n,y_n,z_n) + c\|y_n-y_{n-1}\|^2 \\
&\geq \sum_{i=1}^m f_i(x_{i,n}) +H(y_n) +P(\mathbf{x}_n) -G(\mathbf{x}_n) \\ 
&\quad +\frac{\rho}{2}\left\| \frac{z_n}{\rho} + \mathbf{Ax}_n+By_n-b \right\|^2 -\frac{1}{2\lambda\rho}\left(1 +\frac{\nu\ell_{\psi}}{\omega}\right)\|\nabla H(y_n)\|^2 +\left(c -\frac{(\omega +\nu\ell_{\psi})\nu\ell_{\psi}}{2\lambda\rho}\right)\|y_n-y_{n-1}\|^2 \\
&= \left(\sum_{i=1}^m f_i(x_{i,n}) +P(\mathbf{x}_n) -G(\mathbf{x}_n)\right) +\left(H(y_n) -\frac{1}{2\ell_H}\|\nabla H(y_n)\|^2\right) +\frac{\rho}{2}\left\| \frac{z_n}{\rho} + \mathbf{Ax}_n+By_n-b \right\|^2 \\ 
&\quad +\left(\frac{1}{2\ell_H} -\frac{1}{2\lambda\rho}\left(1 +\frac{\nu\ell_{\psi}}{\omega}\right)\right)\|\nabla H(y_n)\|^2 +\frac{\nu\ell_{\psi}}{2\lambda\rho}(2\ell_H +3\nu\ell_{\psi} -\omega)\|y_n-y_{n-1}\|^2.
\end{align*}
Now, we note that
\begin{align*}
\rho > \frac{\ell_H +\sqrt{\ell_H^2 +8(\ell_H +2\nu\ell_{\psi})^2}}{2\lambda} \geq \frac{\ell_H +\sqrt{\ell_H^2 +8\ell_H^2}}{2\lambda} =\frac{2\ell_H}{\lambda},    
\end{align*}
and so $\frac{1}{2\ell_H} >\frac{1}{\lambda\rho}$. Choosing $\omega =1$ if $\nu\ell_{\psi} =0$, and $\omega =\nu\ell_{\psi}$ if $\nu\ell_{\psi} >0$, it holds that
\begin{align*}
&\theta :=\frac{1}{2\ell_H} -\frac{1}{2\lambda\rho}\left(1 +\frac{\nu\ell_{\psi}}{\omega}\right) \geq \frac{1}{2\ell_H} -\frac{1}{2\lambda\rho}(1 +1) >0 \text{~~and~~} \\
&\frac{\nu\ell_{\psi}}{2\lambda\rho}(2\ell_H +3\nu\ell_{\psi} -\omega) \geq \frac{\nu\ell_{\psi}}{2\lambda\rho}(3\nu\ell_{\psi} -\omega) =\frac{\nu^2\ell_{\psi}^2}{\lambda\rho}. 
\end{align*}
This together with \ref{t:cvg_decrease} yields
\begin{align}\label{eq:L_n_bounded}
L_1 \geq L_n &\geq \left(\sum_{i=1}^m f_i(x_{i,n}) +P(\mathbf{x}_n) -G(\mathbf{x}_n)\right) +\left(H(y_n) -\frac{1}{2\ell_H}\|\nabla H(y_n)\|^2\right) +\frac{\rho}{2}\left\| \frac{z_n}{\rho} + \mathbf{Ax}_n+By_n-b \right\|^2 \notag \\ 
&\quad +\theta\|\nabla H(y_n)\|^2 +\frac{\nu^2\ell_{\psi}^2}{\lambda\rho}\|y_n-y_{n-1}\|^2,
\end{align}
in which the first term on the right hand side is bounded below since $\sum_{i=1}^m f_i(x_i)+P(\mathbf{x})-G(\mathbf{x})$ is coercive, the second term is bounded below due to {the assumption on $H$ and Lemma~\ref{l:diff}\ref{l:diff_bounded}}, while the remaining terms are nonnegative due to their positive coefficients. Therefore, all the terms are bounded. Again, as $\sum_{i=1}^m f_i(x_i)+P(\mathbf{x})-G(\mathbf{x})$ is coercive, the boundedness of the first term implies that $(\mathbf{x}_{n})_{n\in \mathbb{N}}$ is bounded. Since the last two terms are bounded, it follows from \eqref{eq:uppper_zn} that $(z_n)_{n\in \mathbb{N}}$ is bounded. Since $z_{n+1}=z_n+\rho(\mathbf{Ax}_{n+1}+By_{n+1}-b)$ and $B$ has full column rank, we also obtain that $(y_n)_{n\in \mathbb{N}}$ is bounded. Consequently, the sequence $(\mathbf{x}_n,y_n,z_n)_{n\in \mathbb{N}}$ is bounded.

On the other hand, since all the terms on the right hand side of \eqref{eq:L_n_bounded} are bounded below, so is the sequence $(L_n)_{n\in\mathbb{N}}$. Combining with the nonincreasing property of $(L_n)_{n\in\mathbb{N}}$ due to \ref{t:cvg_decrease}, it follows that $(L_n)_{n\in\mathbb{N}}$ is convergent. By rearranging \eqref{eq:decrease} and performing telescoping, we then have that
\begin{align*}
\sum_{n=1}^{+\infty}\delta_y \|y_{n+1}-y_n\|^2 + \delta_{\mathbf{x}}\sum_{n=1}^{+\infty}\|\mathbf{x}_{n+1}-\mathbf{x}_{n}\|^2 \leq \sum_{n=1}^{+\infty}(L_n - L_{n+1}) = L_1 - \lim_{n\to+\infty} L_n < +\infty,
\end{align*}
which implies that $\sum_{n=1}^{+\infty}\|\mathbf{x}_{n+1}-\mathbf{x}_n\|^2 < +\infty$, $\sum_{n=1}^{+\infty}\|y_{n+1}-y_n\|^2 <+\infty$, and by \eqref{eq:zy}, $\sum_{n=1}^{+\infty}\|z_{n+1}-z_n\|^2 < +\infty$. As a result, $\mathbf{x}_{n+1}-\mathbf{x}_{n} \to 0$, $y_{n+1}-y_n \to 0$, $z_{n+1}-z_n \to 0$, and $\mathbf{Ax}_n +By_n =(z_n -z_{n-1})/\rho +b \to b$ as $n \to +\infty$. It then follows from the definition of $L_n$ that the sequence $(\mathcal{L}_\rho(\mathbf{x}_{n},y_{n},z_{n}))_{n\in\mathbb{N}}$ is also convergent with $\lim_{n\to+\infty} \mathcal{L}_\rho(\mathbf{x}_{n},y_{n},z_{n}) =\lim_{n\to+\infty}L_n$.

Let $(\overline{\mathbf{x}},\overline{y},\overline{z})$ be a cluster point of the sequence  $(\mathbf{x}_n,y_n,z_n)_{n\in \mathbb{N}}$. Then there exists a subsequence $(\mathbf{x}_{k_n},y_{k_n},z_{k_n})_{n\in \mathbb{N}}$ that converges to $(\overline{\mathbf{x}},\overline{y},\overline{z})$. In view of \eqref{eq:update_step_x}, letting $n =k_n$ and using Proposition~\ref{p:Bregman}\ref{p:Bregman_3}\&\ref{p:Bregman_4}, we obtain that, for all $i \in \{1,\dots,m\}$ and all $x_i\in \mathbb{R}^{d_i}$,
\begin{align}\label{eq:kn+}
&f_i(x_{i,k_n+1}) +\langle \nabla_i P(\mathbf{x}_{k_n})- g_{i,k_n}, x_{i,k_n+1} -x_i \rangle + \langle z_{k_n}, A_i x_{i,k_n+1} - A_ix_i\rangle +\frac{\rho}{2}\|\mathbf{A}\mathbf{u}_{i,k_n+1}(x_{i,k_n+1}) +By_{k_n} -b \|^2 \notag \\
&\quad  - \frac{\rho}{2}\|\mathbf{A}\mathbf{u}_{i,k_n+1}(x_{i}) +By_{k_n} -b \|^2 +\frac{\mu\alpha}{2}\|x_{i,k_n+1}-x_{i,k_n}\|^2 - \frac{\nu\ell_{\phi}}{2}\|x_{i}-x_{i,k_n}\|^2 \leq f_i(x_{i}).
\end{align}
Since $\nabla P$ is continuous, $\nabla P(\mathbf{x}_{k_n}) \to \nabla P(\overline{\mathbf{x}})$ as $n\to +\infty$. As $G +\frac{\beta}{2}\|\cdot\|^2$ is a continuous convex function, it follows from \cite[Example~9.14]{Rockafellar1998} that $G$ is locally Lipschitz continuous. Since $\mathbf{x}_{k_n} \to \overline{\mathbf{x}}$ as $n\to +\infty$, using Lemma~\ref{l:upsemicont} and passing to a subsequence if necessary, we can and do assume that $g_{k_n}\to \overline{g} \in \partial_L G (\overline{\mathbf{x}})$ as $n\to +\infty$. Now, for each $i \in \{1,\dots,m\}$, letting $x_i =\overline{x}_i$ and $n\to +\infty$ in \eqref{eq:kn+} with noting that $x_{i,k_n+1} \to \overline{x}_i$, $y_{k_n} \to \overline{y}$, and $z_{k_n} \to \overline{z}$, we have that $\limsup_{n \to +\infty} f_i(x_{i,k_n+1}) \leq f_i(\overline{x}_i)$, which together with the lower semicontinuity of $f_i$ yields
\begin{align*}
\lim_{n \to +\infty} f_i(x_{i,k_n+1}) = f_i(\overline{x}_i).
\end{align*}
As $H$, $P$, and $G$ are continuous, it follows that $\lim_{n \to +\infty}\mathcal{L}_\rho(\mathbf{x}_{k_n +1 },y_{k_n+1},z_{k_n +1}) = \mathcal{L}_\rho(\overline{\mathbf{x}},\overline{y},\overline{z})$. Since the sequence $(\mathcal{L}_\rho(\mathbf{x}_{n},y_{n},z_{n}))_{n\in\mathbb{N}}$ is convergent, we deduce that
\begin{align*}
\lim_{n \to +\infty}\mathcal{L}_\rho(\mathbf{x}_{n },y_{n},z_{n}) =\mathcal{L}_\rho(\overline{\mathbf{x}},\overline{y},\overline{z}).    
\end{align*}

We next show that $(\overline{\mathbf{x}},\overline{y},\overline{z})$ is a stationary point of $\mathcal{L}_\rho$. In \eqref{eq:z-update} and \eqref{eq:B'z} with $n$ replaced by $k_n-1$, letting $n\to +\infty$ and using the coninuity of $\nabla H$ and the Lipschitz continuity of $\nabla \psi$, we obtain that $0 =\mathbf{A\overline{x}} + B\overline{y} -b$ and $0 =\nabla H(\overline{y}) + B^\top \overline{z}$. In turn, letting $n \to +\infty$ in \eqref{eq:kn+}, we have for all $i \in \{1,\dots,m\}$ and all $x_i\in \mathbb{R}^{d_i}$ that
\begin{align*}
&f_i(\overline{x}_i) +\langle \nabla_i P(\overline{\mathbf{x}})-\overline{g}_i, \overline{x}_i-x_i \rangle + \langle \overline{z}, A_i\overline{x}_i-A_ix_i \rangle\\
&\quad +\frac{\rho}{2}\|A_1\overline{x}_1 +\dots +A_{i-1}\overline{x}_{i-1} +A_i\overline{x}_i +A_{i+1}\overline{x}_{i+1}+\dots+A_m \overline{x}_{m} +B\overline{y} -b \|^2 \\
&\quad -\frac{\rho}{2}\|A_1\overline{x}_1 +\dots +A_{i-1}\overline{x}_{i-1} +A_ix_i +A_{i+1}\overline{x}_{i+1}+\dots+A_m \overline{x}_{m} +B\overline{y} -b \|^2 -\frac{\nu\ell_{ \phi}}{2}\|x_i-\overline{x}_i\|^2 \leq f_i(x_i),
\end{align*}
which can be written as
\begin{align*}
h_i(\overline{x}_i) \leq h_i(x_i),    
\end{align*}
where $h_i(x_i) =f_i(x_i) +\langle \nabla_i P(\overline{\mathbf{x}}) -\overline{g}_i, x_i \rangle + \langle \overline{z}, A_ix_i \rangle +\frac{\rho}{2}\|A_1\overline{x}_1 +\dots +A_{i-1}\overline{x}_{i-1} +A_ix_i +A_{i+1}\overline{x}_{i+1}+\dots+A_m \overline{x}_{m} +B\overline{y} -b \|^2 +\frac{\nu\ell_{\phi}}{2}\|x_i-\overline{x}_i\|^2$. It follows that, for all $i \in \{1,\dots,m\}$,
\begin{align*}
0 &\in \partial_L h_i(\overline{x}_i) =\partial_L (f_i)(\overline{x}_i) +\nabla_i P(\overline{\mathbf{x}})- \overline{g}_i +A_i^\top \overline{z} +\rho A_i^\top (\mathbf{A}\overline{\mathbf{x}}+B\overline{y}-b).
\end{align*}
Since $\mathbf{A}\overline{\mathbf{x}}+B\overline{y}-b =0$, we obtain that
\begin{align*}
0 \in \partial_L (f_1)(\overline{x}_1)\times \cdots \times \partial_L (f_m)(\overline{x}_m) +\nabla P(\overline{\mathbf{x}}) -\overline{g} +\mathbf{A}^\top \overline{z},
\end{align*}
which completes the proof due to Lemma~\ref{l:cal}\ref{l:cal_sum}\&\ref{l:cal_sep} and the strict differentiability of $G$.
\end{proof}

\begin{remark}[Conditions for the boundedness of the generated sequence] 
In order to establish the boundedness of the sequence $(\mathbf{x}_n,y_n,z_n)_{n\in \mathbb{N}}$, authors in \cite{Guo2016,Sun2017,Tu2019} require that there exists $\sigma\in (0, +\infty)$ such that 
\begin{align*}
\inf_{y\in \mathbb{R}^q} \left(H(y) -\sigma\|\nabla H(y)\|^2\right) >-\infty. 
\end{align*}
This makes it challenging to verify the existence of $\sigma$ in practice. In our analysis in Theorem~\ref{t:cvg}\ref{t:cvg_seq}, we replace this condition with the requirement that $H$ is bounded below, which is much easier to verify.   
\end{remark}

We now establish the convergence of the full sequence generated by Algorithm~\ref{algo:BPLA}. Recall that a proper lower semicontinuous function $\mathcal{F}\colon \mathcal{H}\to \left(-\infty, +\infty\right]$ satisfies the \emph{Kurdyka--Łojasiewicz (KL) property} \cite{Kurdyka1998,Loja63} at $\overline{x} \in \dom \partial_L \mathcal{F}$ if there are $\eta \in (0, +\infty]$, a neighborhood $V$ of $\overline{x}$, and a continuous concave function $\varphi: \left[0, \eta\right) \to \mathbb{R}_+$ such that $\varphi$ is continuously differentiable with $\varphi' > 0$ on $(0, \eta)$, $\varphi(0) = 0$, and, for all $x \in V$ with $
\mathcal{F}(\overline{x}) < \mathcal{F}(x) < \mathcal{F}(\overline{x}) + \eta$, 
\begin{equation*}
\varphi'(\mathcal{F}(x) -\mathcal{F}(\overline{x})) \dist(0, \partial_L \mathcal{F}(x)) \geq 1.   
\end{equation*}
We say that $\mathcal{F}$ is a \emph{KL function} if it satisfies the KL property at any point in $\dom \partial_L \mathcal{F}$. If $\mathcal{F}$ satisfies the KL property at $\overline{x} \in \dom \partial_L \mathcal{F}$, in which the corresponding function $\varphi$ can be chosen as $\varphi(t) = c t ^{1 - \lambda}$ for some $c \in \mathbb{R}_{++}$ and $\lambda \in [0, 1)$, then $\mathcal{F}$ is said to satisfy the \emph{KL property at $\overline{x}$ with exponent $\lambda$}. The function $\mathcal{F}$ is called a \emph{KL function with exponent $\lambda$} if it is a KL function and has the same exponent $\lambda$ at any $x \in \dom \partial_L \mathcal{F}$. Now we will present the proof of the full sequential convergence in the following theorem.

\begin{theorem}[Full sequential convergence]
\label{t:full}
Suppose that Assumption~\ref{a:standing} holds, that $\sum_{i=1}^m f_i(x_i)+P(\mathbf{x})-G(\mathbf{x})$ is coercive, that $H$ is bounded below, that $G$ is differentiable with $\ell_G$-Lipschitz continuous gradient, and that, for all $i\in \{1, \dots, m\}$, $\nabla{\phi_i}$ is $\ell_{\phi}$-Lipschitz continuous. Let $(\mathbf{x}_{n},y_{n},z_{n})_{n\in\mathbb{N}}$ be the sequence generated by Algorithm~\ref{algo:BPLA} with $z_0\in \Ima(B)$ and $\limsup_{n\to +\infty} \mu_n =\overline{\mu} <+\infty$.
Define
\begin{align*}
\mathcal{F}(\mathbf{x},y,z,t):=\mathcal{L}_\rho(\mathbf{x},y,z) +c\|y-t\|^2, \text{~~where~} c=\frac{(\ell_H+2\nu\ell_{\psi})\nu\ell_{\psi}}{\lambda\rho}.
\end{align*}
\begin{enumerate}
\item\label{t:full_cvg} 
Suppose that $\mathcal{F}$ is a KL function. Then the sequence $(\mathbf{x}_{n},y_{n},z_{n})_{n\in\mathbb{N}}$ converges to a stationary point $(\mathbf{x}^{*},y^{*},z^{*})$ of $\mathcal{L}_\rho$ and \begin{align*}
\sum_{n=0}^{+\infty} \|(\mathbf{x}_{n+1},y_{n+1},z_{n+1})-(\mathbf{x}_{n},y_{n},z_{n})\| < +\infty.
\end{align*}
\item\label{t:full_rate} 
Suppose that $\mathcal{F}$ is a KL function with exponent $\kappa\in [0,1)$. Then the following hold:
\begin{enumerate}
\item\label{t:full_rate_finite}
If $\kappa =0$, then $(\mathbf{x}_{n},y_{n},z_{n})_{n\in\mathbb{N}}$ converges to $(\mathbf{x}^{*},y^{*},z^{*})$ in a finite number of steps. 
\item\label{t:full_rate_linear}
If $\kappa\in (0,\frac{1}{2}]$, then there exist $\overline{\gamma}\in \mathbb{R}_{++}$ and $\zeta\in \left(0,1\right)$ such that, for all $n\in \mathbb{N}$, 
\begin{align*}
&\|(\mathbf{x}_{n},y_{n},z_{n})-(\mathbf{x}^{*},y^{*},z^{*})\| \leq \overline{\gamma}\zeta^{\frac{n}{2}}, \\
&\|\mathbf{Ax}_n+By_n-b\| \leq \overline{\gamma} \zeta^{\frac{n}{2}}, \\
&\lvert \mathcal{L}_\rho(\mathbf{x}_{n},y_{n},z_{n}) -\mathcal{L}_\rho(\mathbf{x}^{*},y^{*},z^{*}) \rvert \leq \overline{\gamma}\zeta^n, \\
\text{and~} &|F(\mathbf{x}_n,y_n)-F(\mathbf{x}^{*},y^{*})| \leq \overline{\gamma} \zeta^{\frac{n}{2}}.
\end{align*}
\item\label{t:full_rate_sublinear}
If $\kappa\in (\frac{1}{2},1)$, then there exists $\overline{\gamma}\in \mathbb{R}_{++}$ such that, for all $n\in \mathbb{N}$,
\begin{align*}
&\|(\mathbf{x}_{n},y_{n},z_{n})-(\mathbf{x}^{*},y^{*},z^{*})\| \leq \overline{\gamma} n^{-\frac{1-\kappa}{2\kappa-1}}, \\
&\|\mathbf{Ax}_n+By_n-b\| \leq \overline{\gamma} n^{-\frac{1-\kappa}{2\kappa-1}}, \\
&\lvert \mathcal{L}_\rho(\mathbf{x}_{n},y_{n},z_{n}) -\mathcal{L}_\rho(\mathbf{x}^{*},y^{*},z^{*})\rvert \leq \overline{\gamma} n^{-\frac{2-2\kappa}{2\kappa-1}}, \\
\text{and~} &|F(\mathbf{x}_n,y_n)-F(\mathbf{x}^{*},y^{*})| \leq \overline{\gamma} n^{-\frac{1-\kappa}{2\kappa-1}}.
\end{align*}
\end{enumerate}
\end{enumerate}
\end{theorem}
\begin{proof}
For each $n\in \mathbb{N}$, set $w_n=(\mathbf{x}_{n+1},y_{n+1},z_{n+1},y_n)$ and $\Delta_n =\|\mathbf{x}_{n+2}-\mathbf{x}_{n+1}\| +\|y_{n+2}-y_{n+1}\|$. Let $n \in \mathbb{N}^*$. According to Theorem~\ref{t:cvg}, we have that
\begin{align}\label{eq:auxilary_decrease}
\mathcal{F}(w_{n+1}) +\delta_{\mathbf{x}}\|\mathbf{x}_{n+2}-\mathbf{x}_{n+1}\|^2 + \delta_y \|y_{n+2}-y_{n+1}\|^2 \leq \mathcal{F}(w_{n}),
\end{align}
that $w_{n+1} - w_n \to 0$ as $n\to +\infty$, that $(w_n)_{n\in\mathbb{N}}$ is bounded, that any of its cluster point $\overline{w}=(\overline{\mathbf{x}},\overline{y},\overline{z},\overline{y})$ is a stationary point of \eqref{eq:prob}, and also that $\mathcal{F}(w_{n}) \to \mathcal{F}(\overline{w}) = \mathcal{L}_{\rho}(\overline{\mathbf{x}},\overline{y},\overline{z})$ as $n\to +\infty$. Notice that
\begin{align*}
\Delta_n^2 \leq \left(\frac{1}{\delta_{\mathbf{x}}} +\frac{1}{\delta_y}\right) \left(\delta_{\mathbf{x}}\|\mathbf{x}_{n+2}-\mathbf{x}_{n+1}\|^2 + \delta_y \|y_{n+2}-y_{n+1}\|^2\right),    
\end{align*}
which together with \eqref{eq:auxilary_decrease} yields
\begin{align*}
\mathcal{F}(w_{n+1}) +C_1\Delta_n^2 \leq \mathcal{F}(w_{n}), \text{~~where~} C_1 =\frac{\delta_{\mathbf{x}}\delta_y}{\delta_{\mathbf{x}}+\delta_y}.
\end{align*}

Next, we have that
\begin{align*}
\partial_L ^\mathbf{x} \mathcal{F}(w_n) = \partial_L \left( \sum_{i=1}^{m} f_i\right)(\mathbf{x}_{n+1}) + \nabla P(\mathbf{x}_{n+1}) - \nabla G(\mathbf{x}_{n+1}) + \mathbf{A}^\top z_{n+1} + \rho \mathbf{A}^\top(\mathbf{Ax}_{n+1} + By_{n+1} -b)    
\end{align*}
and, for all $i \in \{1,\dots,m\}$,
\begin{align*}
\partial_L ^{x_i} \mathcal{F}(w_n) = \partial_L f_i(x_{i,n+1})+\nabla_i P(\mathbf{x}_{n+1}) - \nabla_i G(\mathbf{x}_{n+1}) + A_i^\top z_{n+1} + \rho A_i^\top(\mathbf{Ax}_{n+1} + By_{n+1} -b).    
\end{align*}
By revoking the optimality condition for the update of $x_{i,n+1}$ in Step~\ref{step:A2_2} of Algorithm~\ref{algo:BPLA},
\begin{align*}
0 &\in \partial_L f_i(x_{i,n+1}) +\nabla_i P(\mathbf{x}_{n}) -  \nabla_i G(\mathbf{x}_n) + A_i^\top z_n   +\rho A_i^\top(\mathbf{A}\mathbf{u}_{i,n+1}(x_{i,n+1} ) +By_n -b ) \\
&\quad +\mu_n(\nabla \phi_i (x_{i,n+1})- \nabla \phi_i (x_{i,n})),    
\end{align*}
which combined with the above equality implies that
\begin{align*}
&(\nabla_i P(\mathbf{x}_{n+1}) -\nabla_i P(\mathbf{x}_{n})) -(\nabla_i G(\mathbf{x}_{n+1}) -\nabla_i G(\mathbf{x}_n)) +A_i^\top(z_{n+1}-z_n) -\mu_n(\nabla \phi_i (x_{i,n+1}) -\nabla \phi_i (x_{i,n})) \\
&\quad +\rho A_i^\top (A_{i+1}x_{i+1,n+1} +\dots +A_mx_{m,n+1} -A_{i+1}x_{i+1,n} -\dots -A_mx_{m,n} +By_{n+1} -By_n) \in \partial_L ^{x_i} \mathcal{F}(w_n).
\end{align*}
Therefore,
\begin{align}\label{eq:dxi}
\dist(0,\partial_L ^{x_i} \mathcal{F}(w_n)) 
&\leq \|\nabla_i P(\mathbf{x}_{n+1}) -\nabla_i P(\mathbf{x}_{n})\| +\|\nabla_i G(\mathbf{x}_{n+1}) -\nabla_i G(\mathbf{x}_n)\| \notag \\
&\quad +\|A_i^\top(z_{n+1}-z_n)\| +\mu_n\|\nabla \phi_i (x_{i,n+1}) -\nabla \phi_i (x_{i,n})\| \notag \\
&\quad +\rho \|A_i^\top (A_{i+1}x_{i+1,n+1} +\dots +A_mx_{m,n+1} -A_{i+1}x_{i+1,n} -\dots -A_mx_{m,n} +By_{n+1} -By_n\| \notag \\ 
&\leq \ell_P\|\mathbf{x}_{n+1} - \mathbf{x}_{n}\| +\ell_G\|\mathbf{x}_{n} - \mathbf{x}_{n+1}\| +\|A_i^\top\|\|z_{n+1}-z_n\| +\mu_n \ell_{\phi}\|x_{i,n+1}-x_{i,n}\| \notag \\
&\quad +\rho\sum_{j=i+1}^{m} \|A_i^\top\| \|A_j\| \|x_{j,n+1}-x_{j,n}\| +\rho\|A_i^\top\| \|B\| \|y_{n+1}-y_n\| \notag \\
&\leq  \left(\ell_P +\ell_G +\mu_n\ell_{\phi} +\rho\sum_{j=i+1}^{m} \|A_i^\top\| \|A_j\|\right)\|\mathbf{x}_{n+1}-\mathbf{x}_{n}\| +\|A_i^\top\|\|z_{n+1}-z_n\| \notag \\
&\quad +\rho\|A_i^\top\| \|B\| \|y_{n+1}-y_n\|.
\end{align}

We now see that
\begin{align*}
\partial_L ^y \mathcal{F}(w_n) =\nabla H(y_{n+1}) + B^\top z_{n+1} +\rho B^\top(\mathbf{Ax}_{n+1} + By_{n+1} -b) +2c(y_{n+1}-y_{n})    
\end{align*}
and from the optimality condition for the update of $y_{n+1}$ in Algorithm~\ref{algo:BPLA} that 
\begin{align*}
0 &= \nabla H(y_{n+1}) + B^\top z_n +\rho B^\top(\mathbf{Ax}_{n+1} +By_{n+1} -b) +\nu_n(\nabla \psi(y_{n+1}) - \nabla \psi(y_n)).
\end{align*}
This leads to
\begin{align*}
\partial_L ^y \mathcal{F}(w_n) &=B^{\top}(z_{n+1}-z_n) -\nu_n(\nabla \psi(y_{n+1}) -\nabla \psi(y_n)) + 2c(y_{n+1}-y_{n}),   
\end{align*}
and so
\begin{align}\label{eq:dy}
\dist(0,\partial_L ^{y} \mathcal{F}(w_n)) 
&\leq \|B^{\top}(z_{n+1}-z_n)\| +\nu_n\|\nabla \psi(y_{n+1}) -\nabla \psi(y_n)\| + 2c\|y_{n+1}-y_{n})\| \notag \\
&\leq \|B^{\top}(z_{n+1}-z_n)\| +(\nu\ell_\psi +2c)\|y_{n+1}-y_{n})\|.
\end{align}

Since $\partial_L ^z \mathcal{F}(w_n) =\mathbf{Ax}_{n+1} + By_{n+1} -b =\frac{1}{\rho}(z_{n+1} -z_n)$ and
$\partial_L ^t \mathcal{F}(w_n) =-2c(y_{n+1}-y_{n})$, it holds that
\begin{align}\label{eq:dzdt}
\dist(0,\partial_L ^{z} \mathcal{F}(w_n)) =\frac{1}{\rho}\|z_{n+1}-z_n\| \text{~~and~~}
\dist(0,\partial_L ^{t} \mathcal{F}(w_n)) =2c\|y_{n+1}-y_n\|.    
\end{align}
Combining \eqref{eq:dxi}, \eqref{eq:dy}, and \eqref{eq:dzdt}, noting that $\limsup_{n\to +\infty} \mu_n =\overline{\mu} <+\infty$, and using \eqref{eq:BTZy} and \eqref{eq:zy} in the proof of Theorem~\ref{t:cvg}, we derive that there exists $n_0\in \mathbb{N}^*$ and $C_2 \in \mathbb{R}_{++}$ such that, for all $n\geq n_0$, 

\begin{align*}
\dist(0, \partial_L \mathcal{F}(w_n)) 
\leq C_2 \left(\|\mathbf{x}_{n+1}-\mathbf{x}_{n}\| +\|y_{n+1}-y_n\| +\|y_{n}-y_{n-1}\|\right) 
\leq C_2(\Delta_{n-1} +\Delta_{n-2}).
\end{align*}

\ref{t:full_cvg}: We see that all the conditions in the \emph{abstract convergence} framework \cite[Theorem~5.1]{BDL22} are satisfied with $I =\{1,2\}$, $\lambda_1=\lambda_2=1/2$, $\alpha_n \equiv C_1$, $\beta_n \equiv 1/(2C_2)$, and $\varepsilon_n \equiv 0 $. By \cite[Theorem~5.1(i)]{BDL22}, 
\begin{align*}
\sum_{n=0}^{+\infty} (\|\mathbf{x}_{n+2}-\mathbf{x}_{n+1}\|+\|y_{n+2}-y_{n+1}\|) =\sum_{n=0}^{+\infty} \Delta_n < +\infty,
\end{align*}
which implies that $\sum_{n=0}^{+\infty} \|\mathbf{x}_{n+1}-\mathbf{x}_{n}\| <+\infty$ and $\sum_{n=0}^{+\infty} \|y_{n+1}-y_{n}\| < +\infty$. Together with \eqref{eq:zy}, we derive that $\sum_{n=0}^{+\infty} \|z_{n+1}-z_n\|<+\infty$, and hence 
\begin{align*}
\sum_{n=0}^{+\infty} \|(\mathbf{x}_{n+1},y_{n+1},z_{n+1})-(\mathbf{x}_{n},y_{n},z_{n})\| < +\infty,
\end{align*}
which yields the convergence of $(\mathbf{x}_{n},y_{n},z_{n})_{n\in\mathbb{N}}$ to $(\mathbf{x}^{*},y^{*},z^{*})$. In view of Theorem~\ref{t:cvg}\ref{t:cvg_seq}, $(\mathbf{x}^{*},y^{*},z^{*})$ is a stationary point of $\mathcal{L}_\rho$.

\ref{t:full_rate_finite}: This follows from the arguments in the proof of \cite[Theorem~5.1]{BDL22} and \cite[Theorem~2(i)]{Attouch2007}. 

\ref{t:full_rate_linear}: According to \cite[Theorem~5.1(iv)]{BDL22}, there exist $\gamma_0\in \mathbb{R}_{++}$ and $\zeta\in \left(0,1\right)$ such that, for all $n\in \mathbb{N}^*$, 
\begin{align*}
\|(\mathbf{x}_{n},y_{n})-(\mathbf{x}^{*},y^{*})\|\leq\gamma_0 \zeta^{\frac{n}{2}} \text{~~and~~}
|\mathcal{F}(\mathbf{x}_{n},y_{n},z_{n},y_{n-1})-\mathcal{F}(\mathbf{x}^{*},y^{*},z^{*},y^*)| \leq \gamma_0\zeta^n.
\end{align*}
We derive that, for all $n\in \mathbb{N}^*$, $\|\mathbf{x}_n-\mathbf{x}^{*}\|\leq \gamma_0\zeta^{\frac{n}{2}}$ and $\|y_n-y^{*}\|\leq \gamma_0\zeta^{\frac{n}{2}}$, which yields 
\begin{align*}
\|y_n-y_{n-1}\| \leq \|y_n-y^{*}\|+\|y_{n-1}-y^{*}\| \leq (1+\zeta^{-\frac{1}{2}})\gamma_0\zeta^{\frac{n}{2}}.    
\end{align*}
By passing to the limit in \eqref{eq:B'z}, $B^\top z^{*}=-\nabla H(y^{*})$. Following the same steps as in \eqref{eq:BTZy} and \eqref{eq:zy}, we obtain that, for all $n\in \mathbb{N}^*$,
\begin{align*}
\|z_n-z^{*}\| &\leq \frac{\ell_H}{\sqrt{\lambda}}\|y_n-y^{*}\| + \frac{\nu \ell_{\psi}}{\sqrt{\lambda}}\|y_{n}-y_{n-1}\|\leq \gamma_1\zeta^{\frac{n}{2}}, \text{~~where~} \gamma_1 :=\frac{\ell_H\gamma_0+(1+\zeta^{-\frac{1}{2}})\nu\ell_{\psi}\gamma_0}{\sqrt{\lambda}}.
\end{align*}
Consequently,
\begin{align*}
&\|(\mathbf{x}_{n},y_{n},z_{n})-(\mathbf{x}^{*},y^{*},z^{*})\| \leq \|\mathbf{x}_{n}-\mathbf{x}^{*}\| + \|y_{n}-y^{*}\| + \|z_{n}-z^{*}\|\leq (2\gamma_0 +\gamma_1)\zeta^{\frac{n}{2}}, \\
&\|\mathbf{Ax}_n+By_n-b\| =\frac{1}{\rho}\|z_n-z_{n-1}\| \leq \frac{1}{\rho}(\|z_n-z^{*}\| +\|z_{n-1}-z^{*}\|) \leq \frac{(1+\zeta^{-\frac{1}{2}})\gamma_1}{\rho} \zeta^{\frac{n}{2}}, \\
\text{and~} &\|z_n\| \leq \|z_n-z^{*}\| + \|z^{*}\| \leq \gamma_1 \zeta^{\frac{n}{2}} + \|z^{*}\|.
\end{align*}

We now deduce from the definition of $\mathcal{F}$ that
\begin{align*}
|\mathcal{L}_{\rho}(\mathbf{x}_{n},y_{n},z_{n})- \mathcal{L}_{\rho}(\mathbf{x}^{*},y^{*},z^{*})|&= \left|\mathcal{F}(\mathbf{x}_{n},y_{n},z_{n},y_{n-1})-\mathcal{F}(\mathbf{x}^{*},y^{*},z^{*},y^*) - c\|y_n-y_{n-1}\|^2\right|\\
&\leq |\mathcal{F}(\mathbf{x}_{n},y_{n},z_{n},y_{n-1})-\mathcal{F}(\mathbf{x}^{*},y^{*},z^{*},y^*)| + c\|y_n-y_{n-1}\|^2\\
&\leq \gamma_2\zeta^n, \text{~~where~} \gamma_2 :=\gamma_0 + c(1+\zeta^{-\frac{1}{2}})^2\gamma_0^2
\end{align*}
and from the definition of $\mathcal{L}_\rho$ that 
\begin{align*}
|F(\mathbf{x}_n,y_n)-F(\mathbf{x}^{*},y^{*})| &= |\mathcal{L}_{\rho}(\mathbf{x}_n,y_n,z_n)-\mathcal{L}_{\rho}(\mathbf{x}^{*},y^{*},z^{*})-\langle z_n, \mathbf{Ax}_n + By_n-b\rangle - \frac{\rho}{2}\|\mathbf{Ax}_n + By_n-b\|^2| \\
&\leq |\mathcal{L}_{\rho}(\mathbf{x}_n,y_n,z_n)-\mathcal{L}_{\rho}(\mathbf{x}^{*},y^{*},z^{*})| + \|z_n\|\|\mathbf{Ax}_n + By_n-b\| + \frac{\rho}{2}\|\mathbf{Ax}_n + By_n-b\|^2  \\ 
&\leq \gamma_2\zeta^n + \left(\gamma_1 \zeta^{\frac{n}{2}} + \|z^{*}\|\right)\frac{(1+\zeta^{-\frac{1}{2}})\gamma_1}{\rho} \zeta^{\frac{n}{2}}+ \frac{(1+\zeta^{-\frac{1}{2}})^2\gamma_1^2}{2\rho}\zeta^n \\
&\leq \left(\gamma_2 +\frac{(3+\zeta^{-\frac{1}{2}})(1+\zeta^{-\frac{1}{2}})\gamma_1^2}{2\rho} +\frac{\|z^{*}\|(1+\zeta^{-\frac{1}{2}})\gamma_1}{\rho}\right)\zeta^{\frac{n}{2}},
\end{align*}
where the last inequality follows from the fact that $\zeta^n \leq \zeta^{\frac{n}{2}}$ since $\zeta \in (0,1)$. By letting 
\begin{align*}
\overline{\gamma} =\max \left\{ 2\gamma_0+\gamma_1, \frac{(1+\zeta^{-\frac{1}{2}})\gamma_1}{\rho}, \gamma_2 +\frac{(3+\zeta^{-\frac{1}{2}})(1+\zeta^{-\frac{1}{2}})\gamma_1^2}{2\rho} +\frac{\|z^{*}\|(1+\zeta^{-\frac{1}{2}})\gamma_1}{\rho} \right\} 
\end{align*}
and increasing it if necessary, we arrive at the conclusion.

\ref{t:full_rate_sublinear}: Following the arguments in \cite[Theorem~5.1(iv)]{BDL22} and \cite[Theorem~2(iii)]{Attouch2007}, we find $\gamma_0 \in \mathbb{R}_{++}$ such that, for all $n\in \mathbb{N}^*$, 
\begin{align*}
\|(\mathbf{x}_{n},y_{n})-(\mathbf{x}^{*},y^{*})\|\leq \gamma_0 n^{-\frac{1-\kappa}{2\kappa-1}} \text{~~and~~}
|\mathcal{F}(\mathbf{x}_{n},y_{n},z_{n},y_{n-1})-\mathcal{F}(\mathbf{x}^{*},y^{*},z^{*},y^*)| \leq \gamma_0 n^{-\frac{2-2\kappa}{2\kappa-1}}.
\end{align*}
Similar to \ref{t:full_rate_linear}, for all $n\geq 2$, since $n-1\geq \frac{1}{2}n$ and $\frac{1-\kappa}{2\kappa-1} >0$, we derive that 
\begin{align*}
&\|(\mathbf{x}_{n},y_{n},z_{n})-(\mathbf{x}^{*},y^{*},z^{*})\| \leq (2\gamma_0+\gamma_1) n^{-\frac{1-\kappa}{2\kappa-1}},\\
&\|\mathbf{Ax}_n+By_n-b\|\leq \frac{(1 +2^{\frac{1-\kappa}{2\kappa-1}})\gamma_1}{\rho} n^{-\frac{1-\kappa}{2\kappa-1}},\\
&|\mathcal{L}_{\rho}(\mathbf{x}_{n},y_{n},z_{n})- \mathcal{L}_{\rho}(\mathbf{x}^{*},y^{*},z^{*})|\leq \gamma_2 n^{-\frac{2-2\kappa}{2\kappa-1}},
\end{align*}
and that
\begin{align*}
|F(\mathbf{x}_n,y_n)-F(\mathbf{x}^{*},y^{*})| &\leq \gamma_2 n^{-\frac{2-2\kappa}{2\kappa-1}}+ \left(\gamma_1 n^{-\frac{1-\kappa}{2\kappa-1}} + \|z^{*}\|\right)\frac{(1 +2^{\frac{1-\kappa}{2\kappa-1}})\gamma_1}{\rho} n^{-\frac{1-\kappa}{2\kappa-1}} + \frac{(1 +2^{\frac{1-\kappa}{2\kappa-1}})^2\gamma_1^2}{2\rho} n^{-\frac{2-2\kappa}{2\kappa-1}}\\
&\leq \left(\gamma_2 +\frac{(3 +2^{\frac{1-\kappa}{2\kappa-1}})(1 +2^{\frac{1-\kappa}{2\kappa-1}})\gamma_1^2}{2\rho} +\frac{\|z^{*}\|(1 +2^{\frac{1-\kappa}{2\kappa-1}})\gamma_1}{\rho} \right)n^{-\frac{1-\kappa}{2\kappa-1}},
\end{align*}
where $\gamma_1 :=\frac{\ell_H\gamma_0 +(1 +2^{\frac{1-\kappa}{2\kappa-1}})\nu\ell_{\psi}\gamma_0}{\sqrt{\lambda}}$, $\gamma_2 :=\gamma_0 +c(1 +2^{\frac{1-\kappa}{2\kappa-1}})^2\gamma_0^2$, and the last inequality follows from the fact that $n^{-\frac{2-2\kappa}{2\kappa-1}}<n^{-\frac{1-\kappa}{2\kappa-1}}$ for $\kappa\in(\frac{1}{2},1)$. By setting $\overline{\gamma}$ in the same way as in \ref{t:full_rate_linear}, we obtain the conclusion.
\end{proof}

In Theorem~\ref{t:full}, to obtain the convergence of the full sequence generated by Algorithm~\ref{algo:BPLA}, we require that $\mathcal{F}(\mathbf{x},y,z,t) :=\mathcal{L}_\rho(\mathbf{x},y,z) +\frac{(\ell_H+2\nu\ell_{\psi})\nu\ell_{\psi}}{\lambda\rho}\|y-t\|^2$ is a KL function. It is worthwhile mentioning that if the objective function $F(\mathbf{x},y)$ is a semi-algebraic function, then so is $\mathcal{F}(\mathbf{x},y,z,t)$, and hence $\mathcal{F}(\mathbf{x},y,z,t)$ is a KL function with exponent $\kappa \in [0, 1)$; see, e.g., \cite[Example~1]{Attouch2007}.

\section{Numerical results}
\label{sec:numerical_result}

In this section, we provide the numerical results of our proposed algorithm for two case studies: RPCA with modified regularization, and an DC-OPF problem which considers optimal photovoltaic systems placement. All of the experiments are performed in MATLAB R2021b on a 64-bit laptop with Intel(R) Core(TM) i7-1165G7 CPU (2.80GHz) and 16GB of RAM.

\subsection{RPCA with modified regularization}
\label{case:RPCA}

We consider the following RPCA model introduced in Example~\ref{ex:RPCA}
\begin{align}\label{problem:RPCA_org}
    \min_{L,S,T \in \mathbb{R}^{m\times d}} \|L\|_{*} + \tau\|S\|_1  + \frac{\gamma}{2}\|T-M\|_F^2 \quad \text{s.t.~} T=L+S.
\end{align}
We modify the problem \eqref{problem:RPCA_org} as follows
\begin{align}\label{problem:RPCA}
    \min_{L,S,T \in \mathbb{R}^{m\times d}} \|L\|_{*} + \tau\|S\|_1 - \tau \|S\| + \frac{\gamma}{2}\|T-M\|_F^2 \quad \text{s.t.~} T=L+S,
\end{align}
where $\|\cdot\|_1 -  \|\cdot\|$ is the modified regularization. In order to use Algorithm~\ref{algo:BPLA}, we let $f_1(L)= \|L\|_{*}$, $f_2(S)= \tau \|S\|_1$, $H(T)=\frac{\gamma}{2}\|T-M\|_F^2$, $P\equiv 0$, $G(S)=\tau \|S\|$, $D_{\phi_1}(L,L_n)=\frac{\alpha}{2}\|L-L_n\|_F^2$, and $D_{\phi_2}(S,S_n)=\frac{\alpha}{2}\|S-S_n\|_F^2$ . The coefficient matrices are defined by $A_1=A_2=\mathbf{I}_{m\times m}$, and $B=-\mathbf{I}_{m\times m}$. Note that $b=0$ in this case, and for all $n\in\mathbb{N}$ we fix $\mu_n=1,~\nu_n=0$. The problem \eqref{problem:RPCA} can be solved by Algorithm~\ref{algo:BPLA} with the following steps
\begin{align*}
\begin{cases}
    L_{n+1} = \argmin_{L \in \mathbb{R}^{m\times d}}~ \|L\|_{*} - \langle \tau g_{1,n}, L \rangle + \langle Z_n, L \rangle + \frac{\rho}{2}\|L + S_n - T_n\|_F^2 + \frac{\alpha}{2}\|L-L_n\|^2_F,\\
    S_{n+1}= \argmin_{S\in \mathbb{R}^{m\times d}}~ \tau\|S\|_1 - \langle \tau g_{2,n}, S \rangle + \langle Z_n, S \rangle + \frac{\rho}{2}\|L_{n+1} + S - T_n\|_F^2+ \frac{\alpha}{2}\|S-S_n\|^2_F, \\
     T_{n+1} = \argmin_{T\in \mathbb{R}^{m\times d}}~ \frac{\gamma}{2}\|T-M\|^2_F - \langle Z_n, T \rangle + \frac{\rho}{2}\|L_{n+1} + S_{n+1} - T\|^2_F,\\
     Z_{n+1}= Z_n + \rho (L_{n+1} + S_{n+1} - T_{n+1}),
\end{cases}
\end{align*}
where $g_n=(g_{1,n},g_{2,n}) \in \partial_L \|\cdot\| ((L_n,S_n))$. In this case, since $G(S)=\tau \|S\|$, $g_{1,n}$ is a zero matrix with size $m\times d$. This updating scheme can be rewritten as
\begin{align*}
\begin{cases}
    L_{n+1} = \argmin_{L \in \mathbb{R}^{m\times d}} ~ \frac{1}{\rho+\alpha}\|L\|_{*} + \frac{1}{2}\left\|L - \left(\frac{-Z_n-\rho S_n +\rho T_n +\alpha L_n}{\rho+\alpha} \right)    \right\|_F^2,\\
    S_{n+1}=\argmin_{S \in \mathbb{R}^{m\times d}} ~ \frac{\tau}{\rho+\alpha}\|S\|_1 + \frac{1}{2}\left\|S - \left(\frac{\tau g_{2,n}-Z_n -\rho L_{n+1} +\rho T_n +\alpha S_n}{\rho+\alpha} \right)    \right\|_F^2,\\
    T_{n+1}=\frac{\gamma M +Z_n +\rho L_{n+1}+\rho S_{n+1}}{\gamma+\rho},\\
    Z_{n+1}= Z_n + \rho (L_{n+1} + S_{n+1} - T_{n+1}).
\end{cases}
\end{align*}
Both of the $L$ and $S$ subproblems have closed-form solutions, which are
\begin{align*}
    &L_{n+1}= \mathcal{P}_{\frac{1}{\rho+\alpha}} \left(\frac{-Z_n-\rho S_n +\rho T_n +\alpha L_n}{\rho+\alpha}\right),~S_{n+1}=\mathcal{S}_{\frac{\tau}{\rho+\alpha}}\left(\frac{\tau g_{2,n}-Z_n -\rho L_{n+1} +\rho T_n +\alpha S_n}{\rho+\alpha}\right),
\end{align*}
where $\mathcal{P}_c(\cdot)$ denotes the \emph{soft shrinkage} operator imposed on the singular values of the input matrix \cite{Cai2010}, $\mathcal{S}_c(\cdot)$ denotes the \emph{soft shrinkage} operator imposed on all entries of the input matrix \cite{RPCA}, and $c$ is the threshold value. The subgradient $g_{2,n}$ is given by $u_1 v_1^\top$, where $u_1$ and $v_1$ are the first left and right singular column vectors which are obtained via the singular value decomposition of $S_n$. We randomly generate the ground truth matrices for this case study by the procedures described in Appendix~\ref{appen:secA1}. We set $\tau=1/\sqrt{\max(m,d)}$, $\alpha=10^{-2}$ , and $\gamma=1$ for all test cases. To apply Algorithm~\ref{algo:BPLA}, we choose $\rho=2 + 10^{-10}$ (since $\lambda =1,~\ell_H = \gamma$). We compare our proposed algorithm with the three-block ADMM (ADMM-3) used in \cite{RPCA}, with their corresponding Lagrangian penalty term $\rho=2$. The ADMM-3 solves problem \eqref{problem:RPCA_org} while our algorithm solves problem \eqref{problem:RPCA}. The noise parameter $\Gamma$ is set to $10^{-2}$ and $2\times 10^{-2}$, respectively. The algorithms are terminated when the relative change between the two consecutive iterates is less than $10^{-6}$, specifically
\begin{align*}
    \frac{\|(L_{n+1},S_{n+1},T_{n+1})-(L_{n},S_{n},T_{n})\|_F}{\|(L_{n},S_{n},T_{n})\|_F+1} \leq 10^{-6}.
\end{align*}
We denote by $(\widehat{L}_n,\widehat{S}_n,\widehat{T}_n)$ the solution found at each iteration, and by $(L_O,S_O,T_O)$ the ground truth matrices that we want to recover by solving the problem \eqref{problem:RPCA}. The following relative error (RE) with respects to the ground truth is used to measure the quality of the solution
\begin{align*}
    \text{RE}:=\frac{\|(\widehat{L}_n,\widehat{S}_n,\widehat{T}_n)-(L_O,S_O,T_O)\|_F}{\|(L_O,S_O,T_O)\|_F +1}.
\end{align*}
All algorithms are run for 30 times for each test cases with a maximum number of iterations of 4000. At each time, the initial matrices $L$ and $S$ are created randomly, $T$ is initialized by the value of $M$, and $Z$ is initialized as a zero matrix. We perform our experiments on both square ($100\times 100$ and $1000\times 1000$ matrices) and rectangular matrices (here we arbitrary choose $200\times100$ and $2000 \times 1000$ matrices). The average CPU time (in seconds), RE, rank, and sparsity of the solutions are presented in Table~\ref{tab:RPCA_square} and Table~\ref{tab:RPCA_rec}, respectively. The tuple $(r,s)$ in the tables shows the rank and sparsity ratio of the randomly-generated ground truth. Although the proposed algorithm takes more time to run (mainly due to the singular value decomposition process required to compute the subgradient at each iteration), it outperforms ADMM-3 in terms of solution quality for both test cases. Interestingly, the sparsity of the solutions found by our algorithm is also closer to the ground truth values, which indicates the efficiency of the modified regularization. Our algorithm can successfully recover the rank of the ground truth matrices in all test cases for $\Gamma =10^{-2}$. Also, our algorithm can still recover better sparsity level in almost all test cases. When the size of the test matrices is larger, the CPU time of our algorithm is still comparable to that of the ADMM-3, and it still outperforms the ADMM-3 in terms of solution quality. When $\Gamma=2\times10^{-2}$ for the larger case, the rank and sparsity are more challenging to be recovered, and our proposed algorithm can still manage to obtain better sparsity level in almost all test cases.

\begin{center}

\scriptsize
\begin{longtable}[H]{c|lllllllll}

\caption{RPCA results on random generated square matrices}\label{tab:RPCA_square}\\
\hline
\multicolumn{10}{c}{Case $100 \times 100$}   \\
\hline

 $\Gamma$         & $(r, s)$ & Algorithms & Time & RE                  & Iteration    & Rank $\widehat{L}$ & Sparsity $\widehat{S}$ & Rank $L_O$ & Sparsity $S_O$ \\
 \hline
$10^{-2}$        & (10, 0.05) & ADMM-3            & 0.10 & 1.3946E-02          & 74        & 10                 & 543                    & 10         & 500            \\
                         &            & \textbf{BPL-ADMM} & 0.18 & \textbf{1.3920E-02} & 76        & 10                 & \textbf{541}           & 10         & 500            \\
                         & (10, 0.1)  & ADMM-3            & 0.13 & 1.7108E-02          & 101       & 10                 & 1125                   & 10         & 1000           \\
                         &            & \textbf{BPL-ADMM} & 0.23 & \textbf{1.7068E-02} & 103       & 10                 & \textbf{1121}          & 10         & 1000           \\
                         & (15, 0.05) & ADMM-3            & 0.11 & 1.3751E-02          & 90        & 15                 & 755                    & 15         & 500            \\
                         &            & \textbf{BPL-ADMM} & 0.21 & \textbf{1.3715E-02} & 93        & 15                 & \textbf{753}           & 15         & 500            \\
                         & (15, 0.1)  & ADMM-3            & 0.18 & 1.8733E-02          & 139       & 15                 & 1443                   & 15         & 1000           \\
                         &            & \textbf{BPL-ADMM} & 0.33 & \textbf{1.8710E-02} & 143       & 15                 & \textbf{1434}          & 15         & 1000           \\
                         & (20, 0.05) & ADMM-3            & 0.22 & 1.6311E-02          & 170       & 20                 & 1199                   & 20         & 500            \\
                         &            & \textbf{BPL-ADMM} & 0.41 & \textbf{1.6291E-02} & 179       & 20                 & \textbf{1196}          & 20         & 500            \\
                         & (20, 0.1)  & ADMM-3            & 0.34 & 2.1198E-02          & 269       & 20                 & 1973                   & 20         & 1000           \\
                         &            & \textbf{BPL-ADMM} & 0.61 & \textbf{2.1191E-02} & 284       & 20                 & \textbf{1969}          & 20         & 1000           \\
\hline
$2\times10^{-2}$ & (10, 0.05) & ADMM-3            & 0.10 & 1.3907E-02          & 62        & 10                 & 595                    & 10         & 500            \\
                         &            & \textbf{BPL-ADMM} & 0.15 & \textbf{1.3884E-02} & 64        & 10                 & \textbf{591}           & 10         & 500            \\
                         & (10, 0.1)  & ADMM-3            & 0.13 & 1.7392E-02          & 103       & 10                 & 1197                   & 10         & 1000           \\
                         &            & \textbf{BPL-ADMM} & 0.23 & \textbf{1.7347E-02} & 107       & 10                 & \textbf{1194}          & 10         & 1000           \\
                         & (15, 0.05) & ADMM-3            & 0.10 & 1.3972E-02          & 83        & 15                 & 851                    & 15         & 500            \\
                         &            & \textbf{BPL-ADMM} & 0.19 & \textbf{1.3949E-02} & 85        & 15                 & \textbf{848}           & 15         & 500            \\
                         & (15, 0.1)  & ADMM-3            & 0.18 & 1.8635E-02          & 138       & 15                 & 1561                   & 15         & 1000           \\
                         &            & \textbf{BPL-ADMM} & 0.32 & \textbf{1.8633E-02} & 142       & 15                 & \textbf{1556}          & 15         & 1000           \\
                         & (20, 0.05) & ADMM-3            & 0.22 & 1.6348E-02          & 170       & 20                 & 1295                   & 20         & 500            \\
                         &            & \textbf{BPL-ADMM} & 0.41 & \textbf{1.6335E-02} & 181       & 20                 & \textbf{1291}          & 20         & 500            \\
                         & (20, 0.1)  & ADMM-3            & 0.27 & 2.1559E-02          & 215       & 20                 & 2086                   & 20         & 1000           \\
                         &            & \textbf{BPL-ADMM} & 0.48 & \textbf{2.1557E-02} & 225       & 20                 & \textbf{2080}          & 20         & 1000 	       \\
\hline
\multicolumn{10}{c}{Case $1000 \times 1000$}   \\
\hline

$10^{-2}$        & (10, 0.05) & ADMM-3            & 23.99 & 2.5864E-03          & 59        & 10                 & 51210                  & 10         & 50000          \\
                         &            & \textbf{BPL-ADMM} & 35.36 & \textbf{2.5860E-03} & 60        & 10                 & \textbf{51205}         & 10         & 50000          \\
                         & (10, 0.1)  & ADMM-3            & 26.66 & 3.6126E-03          & 66        & 10                 & 100118                 & 10         & 100000         \\
                         &            & \textbf{BPL-ADMM} & 38.65 & \textbf{3.6122E-03} & 67        & 10                 & \textbf{100111}        & 10         & 100000         \\
                         & (15, 0.05) & ADMM-3            & 23.57 & 2.2652E-03          & 60        & 15                 & 51903                  & 15         & 50000          \\
                         &            & \textbf{BPL-ADMM} & 36.14 & \textbf{2.2648E-03} & 61        & 15                 & \textbf{51901}         & 15         & 50000          \\
                         & (15, 0.1)  & ADMM-3            & 27.85 & 3.0318E-03          & 67        & 15                 & 100882                 & 15         & 100000         \\
                         &            & \textbf{BPL-ADMM} & 38.18 & \textbf{3.0315E-03} & 68        & 15                 & \textbf{100880}        & 15         & 100000         \\
                         & (20, 0.05) & ADMM-3            & 25.80 & 2.0584E-03          & 61        & 20                 & 52535                  & 20         & 50000          \\
                         &            & \textbf{BPL-ADMM} & 36.78 & \textbf{2.0582E-03} & 62        & 20                 & \textbf{52533}         & 20         & 50000          \\
                         & (20, 0.1)  & ADMM-3            & 27.20 & 2.6930E-03          & 68        & 20                 & 101744                 & 20         & 100000         \\
                         &            & \textbf{BPL-ADMM} & 38.04 & \textbf{2.6926E-03} & 70        & 20                 & \textbf{101737}        & 20         & 100000         \\
\hline
$2\times10^{-2}$ & (10, 0.05) & ADMM-3            & 21.77 & 3.2703E-03          & 61        & 87                 & 151937                 & 10         & 50000          \\
                         &            & \textbf{BPL-ADMM} & 29.04 & \textbf{3.2700E-03} & 62        & 87                 & \textbf{151913}        & 10         & 50000          \\
                         & (10, 0.1)  & ADMM-3            & 27.69 & 4.2791E-03          & 71        & 119                & \textbf{190550}        & 10         & 100000         \\
                         &            & \textbf{BPL-ADMM} & 35.16 & \textbf{4.2786E-03} & 72        & 119                & 190551        & 10         & 100000         \\
                         & (15, 0.05) & ADMM-3            & 25.68 & 2.8604E-03          & 62        & 89                 & 153312                 & 15         & 50000          \\
                         &            & \textbf{BPL-ADMM} & 38.73 & \textbf{2.8600E-03} & 64        & 89                 & \textbf{153303}        & 15         & 50000          \\
                         & (15, 0.1)  & ADMM-3            & 29.72 & 3.6509E-03          & 72        & 122                & 192932                 & 15         & 100000         \\
                         &            & \textbf{BPL-ADMM} & 41.54 & \textbf{3.6505E-03} & 73        & 122                & 192932       & 15         & 100000         \\
                         & (20, 0.05) & ADMM-3            & 25.27 & 2.5759E-03          & 63        & 92                 & 154723                 & 20         & 50000          \\
                         &            & \textbf{BPL-ADMM} & 34.75 & \textbf{2.5755E-03} & 65        & 92                 & \textbf{154721}        & 20         & 50000          \\
                         & (20, 0.1)  & ADMM-3            & 30.00 & 3.2693E-03          & 73        & 125                & 194854                 & 20         & 100000         \\
                         &            & \textbf{BPL-ADMM} & 42.49 & \textbf{3.2690E-03} & 74        & 125                & \textbf{194851}        & 20         & 100000     \\     
\hline

\end{longtable}
\end{center}

\begin{center}
\scriptsize
\begin{longtable}[H]{c|lllllllll}
\caption{RPCA results on random generated rectangular matrices}\label{tab:RPCA_rec}\\
\hline
\multicolumn{10}{c}{Case $200 \times 100$}   \\
\hline

$\Gamma$       & $(r, s)$ & Algorithms & Time & RE                  & Iteration    & Rank $\widehat{L}$ & Sparsity $\widehat{S}$ & Rank $L_O$ & Sparsity $S_O$ \\
\hline
$10^{-2}$        & (10, 0.05) & ADMM-3            & 0.19 & 9.4198E-03          & 93        & 10                 & 1166                   & 10         & 1000           \\
                         &            & \textbf{BPL-ADMM} & 0.24 & \textbf{9.4045E-03} & 96        & 10                 & \textbf{1162}          & 10         & 1000           \\
                         & (10, 0.1)  & ADMM-3            & 0.23 & 1.2068E-02          & 112       & 10                 & 2343                   & 10         & 2000           \\
                         &            & \textbf{BPL-ADMM} & 0.25 & \textbf{1.2055E-02} & 115       & 10                 & \textbf{2341}          & 10         & 2000           \\
                         & (15, 0.05) & ADMM-3            & 0.25 & 1.0152E-02          & 126       & 15                 & 1708                   & 15         & 1000           \\
                         &            & \textbf{BPL-ADMM} & 0.31 & \textbf{1.0142E-02} & 130       & 15                 & \textbf{1705}          & 15         & 1000           \\
                         & (15, 0.1)  & ADMM-3            & 0.37 & 1.2684E-02          & 186       & 15                 & 2978                   & 15         & 2000           \\
                         &            & \textbf{BPL-ADMM} & 0.41 & \textbf{1.2681E-02} & 194       & 15                 & \textbf{2969}          & 15         & 2000           \\
                         & (20, 0.05) & ADMM-3            & 0.40 & 1.0112E-02          & 206       & 20                 & 2278                   & 20         & 1000           \\
                         &            & \textbf{BPL-ADMM} & 0.46 & \textbf{1.0104E-02} & 213       & 20                 & \textbf{2273}          & 20         & 1000           \\
                         & (20, 0.1)  & ADMM-3            & 0.72 & 1.3707E-02          & 364       & 20                 & 4072                   & 20         & 2000           \\
                         &            & \textbf{BPL-ADMM} & 0.85 & \textbf{1.3693E-02} & 378       & 20                 & \textbf{4068}          & 20         & 2000           \\
\hline
$2\times10^{-2}$ & (10, 0.05) & ADMM-3            & 0.19 & 1.0124E-02          & 93        & 10                 & 1495                   & 10         & 1000           \\
                         &            & \textbf{BPL-ADMM} & 0.21 & \textbf{1.0111E-02} & 95        & 10                 & \textbf{1488}          & 10         & 1000           \\
                         & (10, 0.1)  & ADMM-3            & 0.26 & 1.2467E-02          & 126       & 10                 & 2625                   & 10         & 2000           \\
                         &            & \textbf{BPL-ADMM} & 0.28 & \textbf{1.2452E-02} & 129       & 10                 & \textbf{2624}          & 10         & 2000           \\
                         & (15, 0.05) & ADMM-3            & 0.25 & 1.0044E-02          & 129       & 15                 & 1998                   & 15         & 1000           \\
                         &            & \textbf{BPL-ADMM} & 0.31 & \textbf{1.0041E-02} & 140       & 15                 & \textbf{1994}          & 15         & 1000           \\
                         & (15, 0.1)  & ADMM-3            & 0.31 & 1.2819E-02          & 157       & 15                 & 3415                   & 15         & 2000           \\
                         &            & \textbf{BPL-ADMM} & 0.35 & \textbf{1.2803E-02} & 162       & 15                 & \textbf{3410}          & 15         & 2000           \\
                         & (20, 0.05) & ADMM-3            & 0.39 & 1.1589E-02          & 202       & 20                 & 2860                   & 20         & 1000           \\
                         &            & \textbf{BPL-ADMM} & 0.45 & \textbf{1.1585E-02} & 209       & 20                 & \textbf{2856}          & 20         & 1000           \\
                         & (20, 0.1)  & ADMM-3            & 0.59 & 1.5005E-02          & 306       & 20                 & 4585                   & 20         & 2000           \\
                         &            & \textbf{BPL-ADMM} & 0.68 & \textbf{1.4994E-02} & 320       & 20                 & \textbf{4582}          & 20         & 2000             \\
\hline
\multicolumn{10}{c}{Case $2000 \times 1000$}   \\
\hline

$10^{-2}$        & (10, 0.05) & ADMM-3            & 67.14 & 1.9698E-03          & 89  & 10  & 153310          & 10 & 100000 \\
                         &            & \textbf{BPL-ADMM} & 78.36 & \textbf{1.9696E-03} & 90  & 10  & \textbf{153288} & 10 & 100000 \\
                         & (10, 0.1)  & ADMM-3            & 71.03 & 2.6300E-03          & 99  & 10  & 249807          & 10 & 200000 \\
                         &            & \textbf{BPL-ADMM} & 92.44 & \textbf{2.6298E-03} & 100 & 10  & \textbf{249800} & 10 & 200000 \\
                         & (15, 0.05) & ADMM-3            & 80.06 & 1.6739E-03          & 90  & 15  & 156377          & 15 & 100000 \\
                         &            & \textbf{BPL-ADMM} & 98.49 & \textbf{1.6736E-03} & 91  & 15  & \textbf{156374} & 15 & 100000 \\
                         & (15, 0.1)  & ADMM-3            & 74.23 & 2.2267E-03          & 100 & 15  & 254688          & 15 & 200000 \\
                         &            & \textbf{BPL-ADMM} & 99.27 & \textbf{2.2266E-03} & 101 & 15  & \textbf{254683} & 15 & 200000 \\
                         & (20, 0.05) & ADMM-3            & 69.10 & 1.5351E-03          & 75  & 20  & 160409          & 20 & 100000 \\
                         &            & \textbf{BPL-ADMM} & 89.65 & \textbf{1.5348E-03} & 78  & 20  & \textbf{160387} & 20 & 100000 \\
                         & (20, 0.1)  & ADMM-3            & 63.80 & 1.9954E-03          & 83  & 20  & 258402          & 20 & 200000 \\
                         &            & \textbf{BPL-ADMM} & 89.43 & \textbf{1.9950E-03} & 85  & 20  & \textbf{258383} & 20 & 200000 \\
\hline
$2\times10^{-2}$ & (10, 0.05) & ADMM-3            & 61.93 & 3.2736E-03          & 78  & 128 & 572449          & 10 & 100000 \\
                         &            & \textbf{BPL-ADMM} & 78.93 & \textbf{3.2732E-03} & 80  & 127 & \textbf{572427} & 10 & 100000 \\
                         & (10, 0.1)  & ADMM-3            & 62.20 & 3.8801E-03          & 86  & 119 & \textbf{635960} & 10 & 200000 \\
                         &            & \textbf{BPL-ADMM} & 81.42 & \textbf{3.8799E-03} & 87  & 119 & 635962          & 10 & 200000 \\
                         & (15, 0.05) & ADMM-3            & 64.73 & 2.7395E-03          & 95  & 129 & 573857          & 15 & 100000 \\
                         &            & \textbf{BPL-ADMM} & 83.01 & \textbf{2.7393E-03} & 96  & 129 & \textbf{573833} & 15 & 100000 \\
                         & (15, 0.1)  & ADMM-3            & 57.92 & 3.3192E-03          & 87  & 157 & 639129          & 15 & 200000 \\
                         &            & \textbf{BPL-ADMM} & 74.22 & \textbf{3.3191E-03} & 89  & 157 & \textbf{639117} & 15 & 200000 \\
                         & (20, 0.05) & ADMM-3            & 48.22 & 2.5090E-03          & 97  & 131 & 576133          & 20 & 100000 \\
                         &            & \textbf{BPL-ADMM} & 61.67 & \textbf{2.5090E-03} & 98  & 131 & 576133          & 20 & 100000 \\
                         & (20, 0.1)  & ADMM-3            & 61.68 & 2.9402E-03          & 109 & 160 & 640202          & 20 & 200000 \\
                         &            & \textbf{BPL-ADMM} & 76.83 & \textbf{2.9401E-03} & 110 & 160 & \textbf{640191} & 20 & 200000             \\
\hline

\end{longtable}
\end{center}

\subsection{DC-OPF with optimal photovoltaic system placement}
\label{subsec:dcopf}

We revisit the DC-OPF problem with optimal PV allocation based on the formulations given in \cite{Abraham2018,Tan2023}. The details of the variables as well as the parameters are given in Table~\ref{tab:paramsvars} in the Appendix. The formulation is given as follows
\begin{subequations}
\begin{align}
    \min~~ & \left( \sum_{i\in N}Cu_i + \sum_{i\in N}\left(a_i(P_{i}^{G})^2+b_iP_{i}^{G}+c_i\right) \right) \tag{OPF-1}\label{eq:objOPF0}\\
    \text{subject to~~} 
    &P_{i}^{PV}+P_{i}^{G}- \sum_{j \in M_i}b_{ij}(\Theta_i-\Theta_j)\geq D_i, \quad \forall i \in N \label{eq:flowinout}\\
    &\sum_{i\in N}u_i \geq \frac{\sum_{i\in N}D_i}{2\overline{P^{PV}}} \label{eq:penrate}\\
    & b_{ij}(\Theta_i-\Theta_j)\leq \overline{P}, \quad \forall i \in N,~j \in M_i \label{eq:thermal}\\
    &\Theta_i \in [0,2\pi], \quad \forall i \in N\\
     &0\leq P_{i}^{PV}\leq u_i\overline{P^{PV}}, \quad \forall i \in N \label{eq:existX}\\
     &0\leq P_{i}^{G}\leq \overline{P^G}, \quad \forall i \in N \label{eq:bounds}\\
     &u_i \in \{0,1\}, \quad \forall i \in N.
\end{align}
\end{subequations}
The objective function aims at minimizing the installation cost of PV systems and the operation cost of conventional generators. Constraint \eqref{eq:flowinout} ensures that the demand is met by the power flow from both conventional generators and renewable sources. Constraint \eqref{eq:penrate} makes sure that we have enough PV systems to achieve 50 percent renewable penetration rate. Constraint \eqref{eq:thermal} is the limit of the power flow between any two nodes. The remaining constraints define the boundaries of the decision variables. Using the binary relaxation technique used in \cite{Tan2023}, we have that
\begin{align*}
    &
    (\forall i\in N,\quad u_i \in \{0,1\}) \iff \left(\forall i\in N,\ u_i \in [0,1] \text{~and~} \sum_{i\in N} \left(u_i^2-u_i \right)\geq 0\right).
\end{align*}
Taking into account of the above equivalence, a plausible alternative optimization model for the OPF problem with PV is as follows
\begin{subequations}
\begin{align}
    \min~~ & \left(\sum_{i\in N}Cu_i + \sum_{i\in N}\left(a_i(P_{i}^{G})^2+b_iP_{i}^{G}+c_i\right) -\gamma \sum_{i\in N}(u_i^2 - u_i) \right)\tag{OPF-M}\label{eq:objOPF1}\\
    \text{subject to~~} &\eqref{eq:flowinout} \to \eqref{eq:bounds}\\
        & u_i \in [0,1], \quad \forall i \in N.
\end{align}
\end{subequations}
In order to apply our algorithm, we let $\mathbf{x}=(x_1,x_2,\dots,x_{|N|})$, in which $x_i=[P_i^{PV},P_i^G,\Theta_i,u_i]^\top$ denotes the set of variables associated with bus $i$, for $i \in \{1,\dots,|N|\}$. Here we use the notation $|N|$ to denote the cardinality of the set $N$. Similarly, let $f_i=Cu_i + a_i(P_{i}^{G})^2+b_iP_{i}^{G}+c_i$ be the cost function associated with bus $i$ and $G(\mathbf{x})=\gamma \sum_{i\in N}(u_i^2 - u_i)$, we can reformulate the problem \eqref{eq:objOPF1} as follows
\begin{subequations}
    \begin{align}
        \min_{x_i}~~ &\left(\sum_{i \in N}f_i(x_i) -G(\mathbf{x})\right)\tag{OPF-M2}\label{eq:objOPF2}~\text{s.t.~~} \mathbf{Ax} \leq b,
    \end{align}
\end{subequations}
where $A_i$ are coefficient matrices associated with each $x_i$, and $b$ is the vector form by the right hand side of the problem \eqref{eq:objOPF1}. We then introduce slack variables and then penalise the slack variables to transform all inequalities into equalities as in \cite{Abraham2018}. Two worth noting points are that our DC-OPF problem is nonconvex due to the presence of binary variables (and later by the reformulation into difference-of-convex form), while the DC-OPF problem considered in \cite{Abraham2018} is convex, and that our algorithm is more general and can also cover their ADMM version. Let $p$ be the number of inequality constraints, where $p=9|N|+\sum_{i\in N}|M_i| +1$. By letting $y=(y_1,y_2,\dots,y_p)$ be the slack variables, the problem \eqref{eq:objOPF2} can be equivalently reformulated as follows
\begin{align}
\min_{x_i}~~ &\left(\sum_{i \in N}f_i(x_i) -G(\mathbf{x})\right)\tag{OPF-M2'}\label{eq:objOPF2'}~        \text{s.t.~~} \mathbf{Ax} +y = b,~ y \in \mathbb{R}_+^p.
\end{align}
Here, we observe that
\begin{align*}
y \in \mathbb{R}_+^p \iff \dist(y,\mathbb{R}^{p}_{+}) =0 \iff \dist^2(y,\mathbb{R}^{p}_{+}) \leq 0   
\end{align*}
and that $y\mapsto \dist^2(y,\mathbb{R}^{p}_{+})$ is a differentiable function whose gradient $\nabla\dist^2(y,\mathbb{R}^{p}_{+}) =2(y -\operatorname{Proj}_{\mathbb{R}_+^p}(y))$ is $2$-Lipschitz continuous. So, the final relaxed problem is given as follows
\begin{align}
\min_{x_i}~~ &\left(\sum_{i \in N}f_i(x_i) +\frac{\eta}{2}\dist^2(y,\mathbb{R}^{p}_{+})-G(\mathbf{x})\right)\tag{OPF-M3}\label{eq:objOPF3}~\text{s.t.~~} \mathbf{Ax} +y = b,
\end{align}
where $\eta$ is a positive parameter. Now, \eqref{eq:objOPF3} is a special case of problem \eqref{eq:prob}, in which $f_i(x_i)=Cu_i + a_i(P_{i}^{G})^2+b_iP_{i}^{G}+c_i$, $H(y)=\frac{\eta}{2}\dist^2(y,\mathbb{R}^{p}_{+})$, $G(\mathbf{x})=\gamma \sum_{i\in N}(u_i^2 - u_i)$, $P(\mathbf{x})\equiv 0$, $B=\mathbf{I}_{p\times p}$, and $A_i$ is constructed as described in Appendix~\ref{appen:A_iconstruct}. Each function $f_i(x_i)$ can be rewritten in quadratic form as $f_i(x_i)=\frac{1}{2}x_i^\top Q_ix_i +q_i^\top x_i +c_i$, in which
\begin{align*}
    Q_i=\begin{bmatrix} 0 & 0 & 0 & 0  \\ 0 &2a_i& 0 &0  \\ 0& 0& 0& 0 \\ 0& 0& 0& 0  \end{bmatrix},~ q_i=\begin{bmatrix}
        0\\
        b_i\\
        0\\
        C
    \end{bmatrix}.
\end{align*}
Using Algorithm~\ref{algo:BPLA} with $D_{\phi_i}(x_i,x_{i,n})=\frac{\alpha}{2}\|x_i-x_{i,n}\|^2$, where $\alpha=10^{-2}$, and we also fix $\mu_n=1,~\nu_n=0$ for all $n\in \mathbb{N}$, the updates of the variables are given as follows
\begin{align*}
\begin{cases}
    x_{i,n+1}=-\left(Q_i+\rho A_i^\top A_i+\alpha \mathbf{I}_{p\times p}\right)^{-1}\left(q_i-\nabla_iG(\mathbf{x}_n)+\rho A_i^\top\left(\sum_{k=1}^{i-1}A_kx_{k,n+1}+\sum_{k=i+1}^{N}A_kx_{k,n}+ \right. \right.\\
   \left. \left. \qquad \qquad \quad y_n+\frac{z_n}{\rho}-b\right)-\alpha x_{i,n}\right),~i\in \{1,\dots,|N|\},\\
    y_{j,n+1}=\max \left\{0,\left(-\mathbf{Ax}_{n+1}-\frac{z_n}{\rho} +b \right)_j \right\}  + \frac{\rho}{\eta+\rho} \min \left\{0,\left(-\mathbf{Ax}_{n+1}-\frac{z_n}{\rho} +b \right)_j\right\},~j\in\{1,\dots,p\},\\
    z_{n+1}=z_n+\rho(\mathbf{Ax}_{n+1}+y_{n+1}-b).
\end{cases}
\end{align*}
Note that the closed-form update of $y$ can be derived with the same approach used in \cite[Theorem 1]{Abraham2018}. The proposed algorithm is run for a maximum of 4000 iterations and it is terminated when
\begin{align*}
    \frac{\|(\mathbf{x}_{n+1},y_{n+1},z_{n+1})-(\mathbf{x}_{n},y_{n},z_{n})\|}{\|(\mathbf{x}_{n},y_{n},z_{n})\|} < 10^{-5}.
\end{align*}
The test systems considered here are the 14-bus system used in \cite[Case study 4.2]{Tan2023}, and the 141-bus test case taken from MATPOWER \cite{Zimmerman2011}. We set $(\eta, \rho, \gamma)=(900, 1800+10^{-10},80)$ for the 14-bus case, and then $(\eta, \rho, \gamma)=(3000, 6000+10^{-10},80)$ for the 141-bus case. We compare the performance of our proposed algorithm with the \emph{generalized proximal point algorithm} (GPPA) \cite{An2016}, and the \emph{proximal subgradient algorithm with extrapolation} (PSAe) \cite{Tan2023}. We run all algorithms for 30 times, initialized at the lower bound of the variables, for a maximum of 4000 iterations. The results are given in Table~\ref{tab:compareDCOPF}. The objective function values are calculated using \eqref{eq:objOPF0}.

\begin{table}[H]
\centering
\caption{Comparison of GPPA, PSAe, and the proposed algorithm on 30 runs of the DC OPF model}
\label{tab:compareDCOPF}
\footnotesize
\begin{tabular}{c|ccc|ccc} 
\hline
Test case & ~ & 14-bus &~ &~ &141-bus &~ \\
\hline
Algorithm     & GPPA      & PSAe & \textbf{BPL-ADMM} & GPPA      & PSAe & \textbf{BPL-ADMM}  \\ 
\hline
Mean objective function value  & 2.4347		&2.4346 & 2.4350 & -  & - & 9.8256 
   \\
Best objective function value & 2.4346		&2.4346 & 2.4348 &- &- & 9.8147
   \\
Mean iteration number               & 8             & 20  & 404 &- &- &2105     \\
Mean CPU time (seconds)          & 0.13      & 0.17 &  \textbf{0.08} &- &- &640.67    \\
\hline
\end{tabular}
\begin{tablenotes}\footnotesize
\item[*] The sign ``-'' indicates that the algorithm cannot converge after a maximum running time of 2 hours.
\end{tablenotes}
\end{table}

The best solution found in the 14-bus case using this proposed BPL-ADMM algorithm is approximately the same as the one found in \cite{Tan2023}. This is due to the relaxation used during the modelling process. However, it can be seen that the CPU time is shorter, due to the availability of closed-form solutions of the subproblems. This also makes each iteration of the proposed algorithm less computationally expensive than those of the remaining algorithms. Moreover, it can be observed that for the larger test case, the proposed algorithm benefits from the closed-form solutions of the subproblems and can converge to a stationary point within an acceptable running time. The best solutions found by the proposed algorithm are shown in Figure~\ref{fig:14bus} and Figure~\ref{fig:141bus}, respectively.

\begin{figure}[H]
    \centering
    \includegraphics[width=0.6\linewidth]{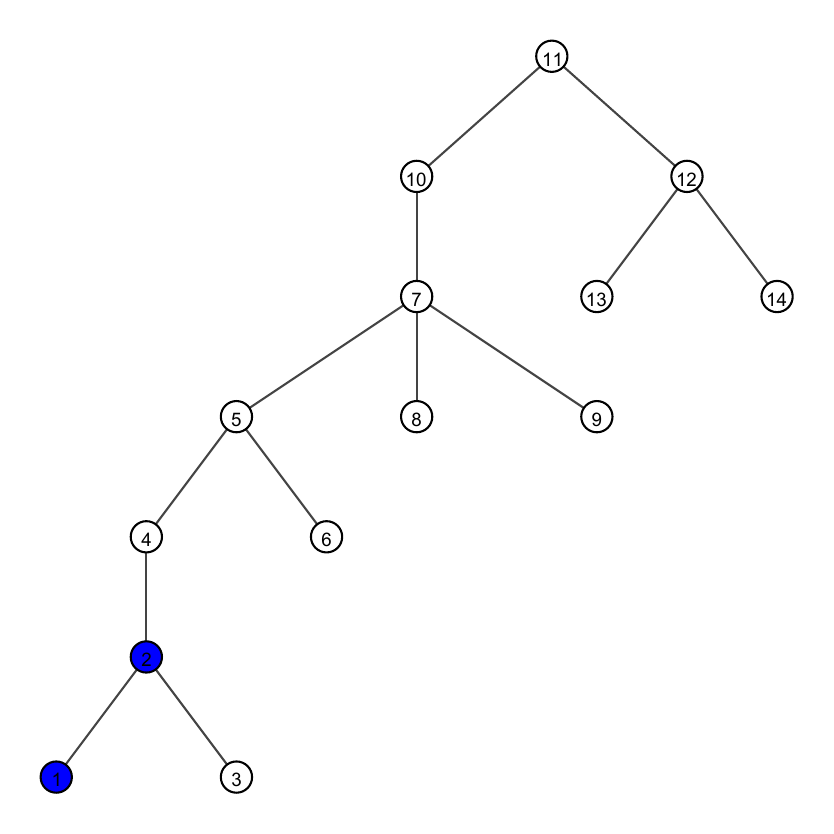}
    \caption{Best solution found for 14-bus case (the blue color indicates that there is a PV system at the bus).}
    \label{fig:14bus}
\end{figure}

\begin{figure}[H]
    \centering
    \includegraphics[width=0.7\linewidth]{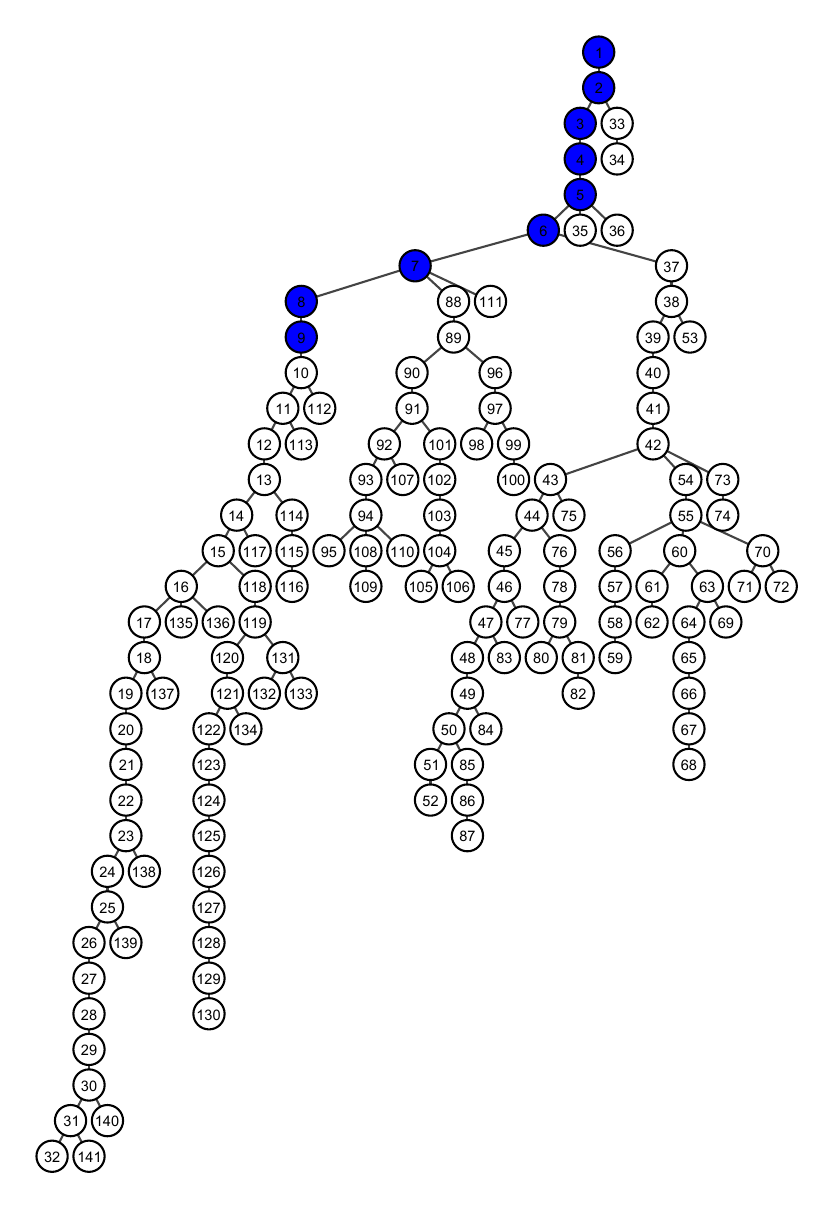}
    \caption{Best solution found for 141-bus case.}
    \label{fig:141bus}
\end{figure}

\section{Conclusion}
\label{sec:conclusion}

In this paper, we have proposed a splitting algorithm with linearization based on ADMM to solve a class of structured nonconvex and nonsmooth optimization problems, which can cover two important classes of problems in the literature. Our proposed structure and the proposed algorithm impose less restrictions on the convexity requirements than some algorithms existing in the current literature. The convergence of the whole sequence generated by our algorithm is proved with some mild additional assumptions and KL property. The efficiency of the proposed algorithm is illustrated on two important nonconvex optimization problems, and it shows competitive results in comparison with the existing algorithms.

\section*{Declarations}

\noindent\textbf{Acknowledgements}. The authors thank Rakibuzzaman Shah and Syed Islam from the Centre for New Energy Transition Research, Federation University Australia for providing the research data and the discussion on the power system optimization problem. They are also grateful to the anonymous reviewers for their insightful feedback during the revision of this manuscript.

\noindent\textbf{Funding}. TNP was supported by Henry Sutton PhD Scholarship Program from Federation University Australia. MND was partially supported by the Australian Research Council (ARC), project number DP230101749, by the PHC FASIC program, project number 49763ZL, and by a public grant from the Fondation Math\'ematique Jacques Hadamard (FMJH).

\noindent\textbf{Conflict of interest}. The authors declare no competing interests.

\noindent\textbf{Availability of data and materials}. All data generated or analyzed during this study are included in this article. In particular, the data for Case study~\ref{case:RPCA} were generated randomly and we explained how they were explicitly generated in Appendix~\ref{appen:secA1}. The data for Case study~\ref{subsec:dcopf} are available in Appendix~\ref{appen:secA2}.

\appendix

\section{Data generation for RPCA model}
\label{appen:secA1}

Let $M=L_O+S_O+N$, where $L_O$ and $S_O$ are the ground truth low-rank and sparse matrices, and $N$ is a Gaussian noise matrix. The following MATLAB code generates the matrix $M$:
\begin{align*}
    &\%\text{Change r and s according to Table 1}\\
    &r=10;~\% \text{rank of ground truth matrix}\\
    &s=0.05;~\% \text{sparsity ratio of ground truth matrix}\\
    &m=100;\\
    &n=100;\\
    &L=\texttt{randn}(m,r)*\texttt{randn}(r,n);\\ &S=\texttt{zeros}(m,n);\\
    &temp=\texttt{randperm}(m*n);\\
    &S(temp(1:\texttt{round}(s*m*n)))=\texttt{randn}(\texttt{round}(s*m*n),1);\\
    &noise =0.01; \% ~\text{Noise level};\\
    &N= \texttt{randn}(m,n)*noise;\\
    &T=L+S;\\
    &M=T+N;
\end{align*}

\section{Data of Case study~\ref{subsec:dcopf}}
\label{appen:secA2}

\begin{center}
\footnotesize
\begin{longtable}[H]{l| p{0.45\columnwidth} |p{0.27\columnwidth}}
\caption{Parameters and variables of the DC-OPF problem}
\label{tab:paramsvars}\\
\hline
\multicolumn{1}{c|}{Parameters}          & \multicolumn{1}{|c}{Description}                                                                                                                                      & \multicolumn{1}{|c}{Values}                                                                                                    \\ \hline
$N$                 & Set of buses                                                                                                                                                          & $\{1,2,\dots,14\}$                                                                                          \\
$M_i$                 & Set of buses that are connected to bus $i$,                                                                                                 &                                                                 \\

$D_i$               & Active power demand at bus $i$                                                                                                                                                     & See \cite[Table 7]{Tan2023}                                                                        \\

$b_{ij}$            & Susceptance value of the line connecting bus $i$ and bus $j$                                                                                                          & See \cite[Table 8]{Tan2023}                                                                                                                                              \\

$C$                 & Unit installation cost of a PV at bus $i$                                                                                                                                  & 1                                                                                                 \\
$a_i,b_i,c_i$             & Coefficients associated with the cost of generator installed at bus $i$. These coefficients for a diesel generator are derived from \cite{Tan2023}. When there is no generator at bus $i$, $a_i=b_i=c_i=0$& $0.246$, $0.084$,  $0.433$ \\
$\overline{P^{PV}}$ & Active power capacity of PVs                                                                                                                                                 & 800 kW                                                                                                     \\

$\overline{P^G_i}$    & Active power capacity of diesel generator at bus $i$. When there is no generator attached to bus $i$, $\overline{P^G_i}=0$                                                                                                                                          & 5000 kW                                                                                                    \\

$\overline{P}$ & Transmission limits of lines                                                                                                                                          & 3000 kW      \\

$\gamma$ & Relaxation parameter & 
\\ \hline
\multicolumn{1}{c|}{Variables}           &                                                                                                                                                                       &                                                                                                           \\ \hline
$P_{i}^{PV}$        & Active power generated by a PV system at bus $i$, $i\in N$ &   \\

$P_{i}^{G}$         & Active power generated by diesel generator at bus $i$, $i \in N$                                                                                                                        &                                                                                                           \\

$u_i $            & 1 if there is a PV system needed at bus $i$, and 0 otherwise, $i\in N$                                                                                                           &                                                                                                           \\

$\Theta_i$          & Voltage angle of bus $i$, $i\in N$                                                                                                                                              &                                                                                                           \\

\hline
\end{longtable}

\end{center}

The values of the parameters in Table~\ref{tab:paramsvars} are used for the 14-bus test case. We use the same parameter notations for the 141-bus test case. The parameters of the 141-bus test case (line configurations, load demand, diesel generator capacity) can be found within MATPOWER software. Note that we also use 800 kW PV for the 141-bus case. All calculations are performed on Per Unit (pu) values.

\section{Construction of matrices $A_i$ in the DC-OPF problem}
\label{appen:A_iconstruct}

We give an example of how the matrices $A_i$ can be constructed. For simplicity, we use a 2-bus example. The procedures are still the same if there are more buses. Consider the example given in Figure~\ref{fig:2bus}.

\begin{figure}[H]
\centering
\includegraphics[scale=0.2]{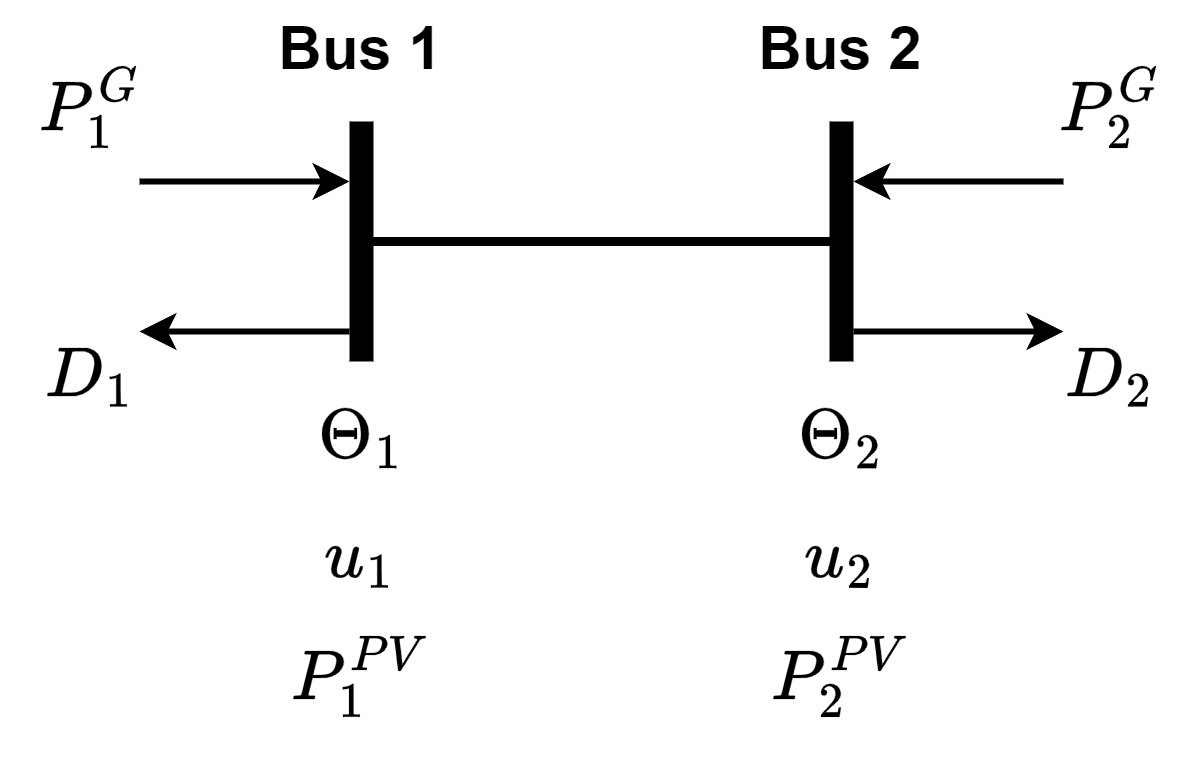}
\caption{Two-bus example.}
\label{fig:2bus}
\end{figure}
The problem \eqref{eq:objOPF3} in this case reads as
    \begin{align*}
        \min_{x_i}~~ &\left(\sum_{i =1}^{2}f_i(x_i) +\frac{\eta}{2}\dist^2(y,\mathbb{R}^{p}_{+})-G(\mathbf{x})\right)\\
        \text{subject to~~} &A_1x_1+A_2x_2 +y = b.
    \end{align*}
The matrices $A_1$, $A_2$, and vector $b$ are given by
\begin{align*}
    A_1= \begin{bmatrix}
-1 & -1 & b_{12} & 0 \\
0 & 0 & -b_{21} & 0 \\
-1 & 0 & 0 & 0 \\
0 & 0 & b_{12} & 0 \\
0 & 0 & -b_{21} & 0 \\
1 & 0 & 0 & -\overline{P^{PV}} \\
-1 & 0 & 0 & 0 \\
0 & 0 & 0 & 0 \\
0 & 0 & 0 & 0 \\
0 & 1 & 0 & 0 \\
0 & -1 & 0 & 0 \\
0 & 0 & 0 & 0 \\
0 & 0 & 0 & 0 \\
0 & 0 & 0 & 1 \\
0 & 0 & 0 & -1 \\
0 & 0 & 0 & 0 \\
0 & 0 & 0 & 0 \\
0 & 0 & 1 & 0 \\
0 & 0 & -1 & 0 \\
0 & 0 & 0 & 0 \\
0 & 0 & 0 & 0 \\
\end{bmatrix}, \quad A_2=\begin{bmatrix}
0 & 0 & -b_{12} & 0 \\
-1 & -1 & b_{21} & 0 \\
-1 & 0 & 0 & 0 \\
0 & 0 & -b_{12} & 0 \\
0 & 0 & b_{21} & 0 \\
0 & 0 & 0 & 0 \\
0 & 0 & 0 & 0 \\
1 & 0 & 0 & -\overline{P^{PV}} \\
-1 & 0 & 0 & 0 \\ 
0 & 0 & 0 & 0 \\
0 & 0 & 0 & 0 \\
0 & 1 & 0 & 0 \\
0 & -1 & 0 & 0 \\
0 & 0 & 0 & 0 \\
0 & 0 & 0 & 0 \\
0 & 0 & 0 & 1 \\
0 & 0 & 0 & -1 \\
0 & 0 & 0 & 0 \\
0 & 0 & 0 & 0 \\
0 & 0 & 1 & 0 \\
0 & 0 & -1 & 0 
\end{bmatrix}, \quad b= \begin{bmatrix}
-D_1 \\
-D_2 \\
-\frac{1}{2}(D_1+D_2) \\
\overline{P} \\
\overline{P} \\
0 \\
0 \\
0 \\
0 \\
\overline{P^G_1} \\
0 \\
\overline{P^G_2} \\
0 \\
1 \\
0 \\
1 \\
0 \\
2\pi \\
0 \\
2\pi \\
0 
\end{bmatrix}  
\end{align*}
By this way of construction, $A_i$ has full column rank.

\end{document}